\documentclass[11pt, reqno]{amsart}
\usepackage[margin=3cm]{geometry} 
\usepackage{graphicx}
\graphicspath{{STLCT_Images/}{./}} % Specifies where to look for included images (trailing slash required)
\usepackage{float}
\usepackage[colorlinks=true, allcolors=blue]{hyperref}
\usepackage{amsmath,amssymb,amsfonts}%
\usepackage{amsthm}%
\usepackage{mathrsfs}%
\usepackage[title]{appendix}%
\usepackage{xcolor}%
\usepackage{textcomp}%
\usepackage{manyfoot}%
\usepackage{booktabs}%
\usepackage[ruled,vlined]{algorithm2e}
\usepackage{listings}%
\usepackage{soul}
\usepackage{blindtext}
\usepackage{enumerate}%
\allowdisplaybreaks
\usepackage{subcaption}  % 

\newtheorem{theorem}{Theorem}[section]

\newtheorem{lemma}[theorem]{Lemma}%  
\newtheorem{remark}[theorem]{Remark}%
\newtheorem{corollary}[theorem]{Corollary}%
\newtheorem{proposition}[theorem]{Proposition}%

\newcommand{\N}{\mathbb{N}}
\newcommand{\Z}{\mathbb{Z}}
\newcommand{\R}{\mathbb{R}}
\newcommand{\C}{\mathbb{C}}

\newcommand{\T}{\mathbb{T}}

\numberwithin{equation}{section}
\setstcolor{red}

\begin{document}
\title[Stable STLCT phase retrieval in Gaussian shift-invariant spaces]{Stable phase retrieval from short-time linear canonical transforms of signals in Gaussian shift-invariant spaces}
\author{Cheng Cheng, Baixiang Wu, Jun Xian}
\thanks{This project is partially supported by  National Key RD Program of China (No. 2024YFA1013703),  National Natural Science Foundation of China (12171490, 12371104), China Scholarship Council (202306380270, 202506380227), and the Guangdong Province Key Laboratory of Computational Science, China.}

\maketitle

\begin{abstract}
Gabor phase retrieval for signals has attracted considerable attention in recent years. For the more general short-time linear canonical transform (STLCT), which arises naturally in optical systems and canonical time--frequency analysis, existing work has so far focused mainly on uniqueness and sampling conditions. Explicit reconstruction formulas, quantitative stability estimates, and robust reconstruction algorithms, however,  are still missing. In this paper, we study uniqueness, stability, and robust reconstruction for phase retrieval from phaseless STLCT measurements in the complex Gaussian shift-invariant space $V_\beta^\infty(\varphi)$.  We first prove that every signal in $V_\beta^\infty(\varphi)$ is uniquely determined, up to a global unimodular constant, by its phaseless STLCT measurements on the semi-discrete set $\frac{\beta}{2}\mathbb Z\times\mathbb R$, and we derive an explicit reconstruction formula. We then establish stability on intervals under an anchor-point condition, showing that the stability constant is governed by the maximal spacing between adjacent anchor points rather than by the radius of the whole interval. This prevents exponential deterioration with respect to the interval size. Motivated by the practical setting in which only finitely many discrete noisy magnitude samples are available, we further develop an explicit reconstruction algorithm with quantitative robustness guarantees, where the reconstruction error is controlled by the discretization parameters, the noise level, and the conditioning induced by the anchor points. In the Fourier case, our results recover the corresponding Gabor phase retrieval results of Grohs and Liehr and provide improved stability constants.
\end{abstract}

\section{Introduction}
		
\medskip 

Phase retrieval concerns the reconstruction of a function $f$ from magnitude-only measurements $\{\, |\psi(f)| : \psi \in \Psi \,\}$, where $\Psi$ is a family of linear functionals acting on a function space. Such problems are inherently nonlinear and arise naturally in many scientific and engineering applications in which detectors record only intensities rather than the underlying signal, including X-ray crystallography \cite{millane1990phase}, coherent diffractive imaging \cite{bunk2007diffractive}, and optics \cite{walther1963question}. From a mathematical point of view, phase retrieval is a nonlinear inverse problem in which one seeks to recover a function from quadratic measurements. Its central questions concern injectivity of the associated nonlinear measurement map, stability of the inverse problem, and the design of constructive reconstruction procedures, see \cite{fannjiang2020numerics,grohs2020phase} for surveys, as well as \cite{alaifari2017phase,alaifari2021gabor,alaifari2021uniqueness,balan2006signal,cahill2016phase,chen2022phase,cheng2019phaseless,grochenig2020phase,grohs2024stable,grohs2019stable,thakur2011reconstruction} and  references therein.

A classical example is Fourier phase retrieval, corresponding to the choice of $\Psi$ as the Fourier transform. In this setting, the modulus of the Fourier transform does not determine the signal uniquely. Besides the trivial ambiguities of translation, conjugate reflection, and multiplication by unimodular constants, there exist infinitely many essentially different functions sharing the same Fourier magnitude \cite{akutowicz1956determination,beinert2015ambiguities,hurt1989phase,walther1963question}. This shows that taking the modulus alone is, in general, insufficient for recovering the phase information. It therefore motivates the search for alternative measurement schemes with enough redundancy to guarantee uniqueness and stable recovery.

A prominent alternative is provided by the short-time Fourier transform (STFT) 
$$\mathcal{V}_\omega f(x,\xi)= \int_{\R} f(t)\,\overline{\omega(t-x)}\,e^{-2\pi i \xi t}\,dt,$$
which yields a highly redundant time--frequency representation and has attracted considerable attention in phase retrieval  \cite{alaifari2024connection,alaifari2021uniqueness,alaifari2022phase,bartusel2023injectivity,chen2025sampling,grohs2020phase,grohs2022foundational,grohs2023nonuniqueness,grohs2025completeness,grohs2019stable,wellersho?2024injectivity,wellersho?2024phase,wellersho?2025dense}. When the window function $\omega$ is chosen to be the Gaussian
$$\varphi:=e^{-\frac{\cdot^2}{2\sigma^2}},$$
with variance $\sigma^2$, the transform $\mathcal{V}_\varphi f$ enjoys strong analytic properties and is commonly referred to as the Gabor transform \cite{grochenig2001foundations}. The corresponding phase retrieval problem is known as \emph{Gabor phase retrieval}. Under mild nondegeneracy assumptions on the window, the spectrogram $|\mathcal{V}_\omega f|^2$ determines the signal uniquely up to a unimodular constant \cite{alaifari2021uniqueness,bartusel2023injectivity,grohs2020phase,grohs2019stable}. However, in infinite-dimensional spaces, uniqueness does not in general imply uniform stability: the nonlinear map $f\mapsto |\mathcal{V}_\omega f|$ typically fails to admit a uniform Lipschitz inverse \cite{alaifari2017phase,cahill2016phase}. This has led to the study of conditional stability under additional structural assumptions \cite{alaifari2021gabor,grohs2019stable,grohs2022stable}. 

The shift-invariant space 
$$ V_\beta^p(h) = \Bigl\{ \sum_{k\in\mathbb Z} c_k\,h(\cdot-\beta k):\ \{c_k\}_{k\in\mathbb Z}\in \ell^p(\mathbb Z) \Bigr\},  \ 1\le p\le \infty, $$
is a fundamental class of function spaces in harmonic analysis and sampling theory \cite{aldroubi2001nonuniform,aldroubi2001pframes,aldroubi2005convolution,deboor1994structure,grochenig2018sampling}. They interpolate between finite-dimensional approximation spaces and classical spaces such as Paley--Wiener and spline spaces, and provide a natural framework for sampling and reconstruction problems. For suitable generators, and in particular for the Gaussian, elements of these spaces admit strong analytic structure, which makes them especially amenable to complex-analytic methods.  Phase retrieval in shift-invariant spaces has been extensively studied for real-valued signals via phaseless pointwise sampling \cite{chen2020phase,cheng2021stable,cheng2019phaseless,cheng2021stable2,grochenig2020phase,li2025global,romero2021sign,sun2021local,thakur2011reconstruction}, and for complex-valued signals via structured measurements designed  \cite{chen2024conjugate,cheng2025conjugate,lai2021conjugate,mc2004phase}.

Some particularly relevant recent results are due to Grohs and Liehr \cite{grohs2023injectivity,grohs2024stable}, who studied Gabor phase retrieval for signals in the complex Gaussian shift-invariant space $V_\beta^\infty(\varphi)$. In \cite{grohs2023injectivity}, they proved that every signal $f\in V_\beta^\infty(\varphi)$ is uniquely determined, up to a global unimodular constant,  from the phaseless Gabor lattice measurements $|\mathcal{V}_{\varphi} f(\frac{\beta}{2+\epsilon}\Z, \Z)|$ for any $\epsilon > 0$. 
They later showed in \cite{grohs2024stable} that phaseless Gabor measurements on parallel lines in the time--frequency plane determine the signal uniquely, up to a global unimodular constant, established conditional stability estimates on intervals, and developed an explicit reconstruction procedure from finitely many noisy samples. These results demonstrate that the combination of Gaussian shift-invariant structure and redundant time--frequency measurements leads not only to uniqueness, but also to stable and constructive recovery.

These results naturally raise the question whether the same program can be carried out beyond the Fourier framework. This motivates the study of the short-time linear canonical transform (STLCT), which extends the STFT by replacing the Fourier kernel with the kernel of the linear canonical transform (LCT). The LCT is parametrized by
$$ \mathbf A=
\begin{pmatrix}
a & b\\
c & d
\end{pmatrix}, \ a,b,c,d\in \R,  ad-bc=1,$$
and includes the Fourier transform, the fractional Fourier transform, and the Fresnel transform as special cases \cite{healy2016linear,xu2013linear}. It plays an important role in optics \cite{healy2016linear,stern2008uncertainty} and quantum mechanics \cite{moshinsky1971linear}. For a window function $\omega$, the STLCT is defined by
$$\mathcal S_\omega^{\mathbf A}f(x,\xi) = \int_{\mathbb R} f(t)\overline{\omega(t-x)}\,K_{\mathbf A}(\xi,t)\,dt, $$
where
$$ K_{\mathbf A}(\xi,t) =\frac{1}{\sqrt{ib}}\exp\!\left(\frac{i\pi}{b}(at^2-2\xi t+d\xi^2)\right)$$
is the LCT kernel \cite{kou2012windowed,xu2013linear}.  Existing results on STLCT phase retrieval mainly address uniqueness and sampling conditions \cite{dong2025uniqueness,li2024uniqueness,li2025uniqueness,liu2026uniqueness}. By contrast, explicit reconstruction formulas, quantitative stability estimates, and robust algorithms for infinite-dimensional signal classes remain largely missing. This gap is particularly significant from the practical point of view, since in applications one only has access to finitely many discrete phaseless samples, typically corrupted by noise, and therefore any useful phase retrieval theory must ultimately lead to a robust reconstruction algorithm.

The purpose of this paper is to develop such a theory for Gaussian shift-invariant signals from phaseless STLCT measurements. Using a chirp-modulated Gaussian window, we show that the analytic framework underlying the Gabor case extends to the STLCT setting. First, we prove that every signal in $V_\beta^\infty(\varphi)$ is uniquely determined, up to a global unimodular constant, by its phaseless STLCT measurements on the semi-discrete set $\frac{\beta}{2}\mathbb Z\times\mathbb R$, and we derive an explicit reconstruction formula, see Theorem \ref{pr.gassian.thm1}. Next, we establish conditional stability on bounded intervals. By introducing anchor points at which the signal magnitude is bounded away from zero, we show that the stability constant depends on the maximal gap between adjacent anchor points rather than on the radius of the whole interval, thereby avoiding exponential deterioration with respect to the interval size, see Theorem \ref{stable.thm2}.  Finally, motivated by the practical setting in which only finitely many discrete noisy phaseless STLCT samples are available, we develop an explicit reconstruction algorithm and establish quantitative robustness estimates, see Theorem \ref{alg-stability.thm}. The algorithm first detects anchor points from the sampled magnitude data, then performs local reconstruction on neighboring anchor intervals, and finally propagates the phase information across adjacent anchors to recover the signal globally. Our robustness analysis makes precise how the reconstruction error depends on the discretization parameters, the measurement noise, and the conditioning induced by the anchor point condition. Numerical experiments illustrate the effectiveness of the algorithm. In the Fourier case, our results recover the corresponding Gabor phase retrieval theory of Grohs and Liehr \cite{grohs2024stable}, while also yielding sharper stability bounds.

The remainder of the paper is organized as follows. In Section \ref{prelim.sec}, we collect the necessary preliminaries on the STLCT and Gaussian shift-invariant spaces. In Section \ref{uniq.sec}, we establish the uniqueness result of STLCT phase retrieval for complex-valued signals in Gaussian shift-invariant spaces. Section \ref{stable.sec} is devoted to the corresponding stability analysis of STLCT phase retrieval. Finally, in Section \ref{robust.sec}, we develop the reconstruction algorithm from finitely many discrete phaseless STLCT samples, derive quantitative approximation guarantees, and present numerical experiments illustrating the effectiveness of the proposed method.

    \section{Preliminaries}\label{prelim.sec}

\medskip

In this section, we collect the notation, definitions, and preliminary results that will be used throughout the paper.

\subsection{Time-frequency transforms}

For $f \in L^1 (\R)$, its \emph{Fourier transform} $\mathcal{F} f $ is defined by
	\begin{align*}
		\mathcal{F} f (\xi)  := \int_{\R} f(t) e^{- 2 \pi i \xi t} dt.
	\end{align*}
	For $f \in L^2 (\R)$, its \emph{short-time Fourier transform} (STFT) $\mathcal{V}_\omega$ with a window function $\omega \in L^2 (\R)$ is defined by
	\begin{align*}
		\mathcal{V}_\omega f (x , \xi) := \int_{\R} f(t) \overline{\omega(t - x)} e^{- 2 \pi i \xi t} dt.
	\end{align*}
	Given a parameter matrix $\mathbf{A} = 
	\begin{pmatrix}
		a & b \\
		c & d
	\end{pmatrix}
	$
	such that $a, b, c, d \in \R$ and $ad - bc = 1$. We define the \emph{linear canonical transform} (LCT) \cite{healy2016linear, xu2013linear} $\mathcal{F}_{\textbf{A}}$ of $f \in L^1(\R)$ with parameter matrix \textbf{A} by
	\begin{align*}
		\mathcal{F}_{\textbf{A}} f (\xi) := \left\{
		\begin{array}{ll}
			\int_{\R} f(t) K_{\textbf{A}} (\xi , t) dt, \quad &  b \ne 0; \\
			\sqrt{d} e^{i \pi cd  \xi^2} f(d\xi), \quad & b = 0,
		\end{array}
		\right.
	\end{align*}
	where the transform kernel $K_{\textbf{A}}$ is given by
	\begin{align}\label{eq12}
		K_{\textbf{A}} (\xi , t) := \frac{1}{\sqrt{i b}} e^{\frac{i \pi}{b} (a t^2 - 2\xi t + d\xi^2)}, \ b \ne 0.
	\end{align}

	When $b = 0$, the LCT degenerates into a scaling and chirp multiplication operator. Excluding this degenerate  case,  we assume $b > 0$ throughout this paper, since the case $b<0$ can be handled analogously.

    The {\em short-time linear canonical transform} (STLCT) of $f \in L^2(\R)$ with parameter matrix $\mathbf{A}$ and window function $\omega \in L^2(\R)$ is defined as
\begin{align}\label{stlct.def}
    \mathcal{S}_\omega^{\mathbf{A}} f(x, \xi) := \int_{\R} f(t) \overline{\omega(t - x)} K_{\mathbf{A}}(\xi, t) \, dt,
\end{align}
where $K_{\mathbf{A}}$ is as in \eqref{eq12} \cite{kou2012windowed, xu2013linear}. When $\mathbf{A} = \begin{pmatrix} 0 & 1 \\ -1 & 0 \end{pmatrix}$, the STLCT reduces to the classical STFT.

    To establish the relationship between STLCT and STFT, we introduce the following auxiliary functions. 
    Given the parameter matrix $\mathbf{A}$ and two complex-valued functions $f, \omega:\R \to \C$  along with two constants $x\in \R$ and $\xi\in \R$,  we define the following functions 
    \begin{subequations}\label{def.auxi}
\begin{align}
&\check{f}(t)= e^{i\pi \frac{a}{b} t^2}f(t), \label{def.auxia}\\
& f^{\omega,x}(t)= f(t)\overline{\omega(t-x)} e^{i\pi \frac{a}{b} t^2},\label{def.auxib} \\
&f_\xi(t)=f(t-\xi)\overline{f(t)}\label{def.auxic},
    \end{align}
\end{subequations}       
    where $t\in \R$. We refer to $\check{f}$ as the \emph{chirp modulation} of $f$, and $f_\xi$ as the \emph{tensor product} of $f$.
    These auxiliary functions allow us to express the STLCT in terms of the STFT:
    \begin{align}\label{eq59}
        \mathcal{S}_\omega^{\mathbf{A}}f (x , \xi) = \frac{1}{\sqrt{i b}} e^{i\pi \frac{d}{b} \xi^2} \mathcal{V}_\omega\check{f} (x , \frac{\xi}{b})  = \frac{1}{\sqrt{i b}} e^{i\pi \frac{d}{b} \xi^2} \mathcal{F} f^{\omega,x} (\frac{\xi}{b}).
    \end{align}
    The tensor product $f_\xi$ in \eqref{def.auxic} plays a crucial role in our analysis, as it captures the phase information of $f$ through the phaseless STLCT measurements, see  Section~\ref{uniq.sec}.

\subsection{Gaussian shift-invariant spaces}
    We now introduce the shift-invariant spaces that will be the signal class of interest. Let $\varphi:= e^{-\frac{\cdot^2}{2\sigma^2}}$ be a Gaussian function with variance $\sigma^2$. The {\em Gaussian shift-invariant space} is defined as
\begin{align}\label{gauss.sis.def}
    V_{\beta}^p(\varphi) := \Big\{ \sum_{k \in \Z} c_k \varphi(\cdot - \beta k) : \{c_k\}_{k\in \Z} \in \ell^p(\Z) \Big\}, \quad 1 \le p \le \infty,
\end{align}
where $\beta > 0$ is the step-size.
It is well known that $ V_{\beta}^p(\varphi) \subset V_{\beta}^q(\varphi) \subset L^q(\R)$ for $1 \le p \le q \le \infty$ \cite{aldroubi2001nonuniform}.
    Define the $\frac{1}{\beta}$-periodic map $\Psi_{\beta} : \R \to \R$ via
    \begin{align}\label{psi.beta.eq}
        \Psi_{\beta} (t) {}& = \sum_{n \in \Z} |\mathcal{F} \varphi (t + \frac{n}{\beta}) |^2.
    \end{align}
    One may easily verify that  $\{\varphi(\cdot - \beta n): n\in \Z \}$ forms a Riesz basis for $V_{\beta}^2(\varphi)$, i.e., there exist constants $0<A_2\le B_2<\infty$ such that
    \begin{align}\label{eq90}
        A_2\|c\|_{\ell^2(\Z)} \le \left\| \sum_{n\in \Z} c_n \varphi(\cdot - \beta n) \right\|_{L^2(\R)} \le B_2\|c\|_{\ell^2(\Z)},
    \end{align}
    cf. Proposition \ref{prop3}, where we only need verify that $\beta A_2 \le \Psi_{\beta}(t) \le \beta B_2$ holds for almost every $t\in [0,\frac{1}{\beta}]$.
    
    %Furthermore, for $c\in \ell^p(\Z)$ with some $1\le p\le \infty$, there exist $0<A_p\le B_p< \infty$ such that
    %\begin{align}\label{eq90}
    %    A_p\|c\|_{\ell^p(\Z)} \le \left\| \sum_{n\in \Z} c_n \varphi(\cdot - \beta n) \right\|_{L^p(\R)} \le B_p\|c\|_{\ell^p(\Z)}
    %\end{align}
    %\cite[Theorem 2.4]{aldroubi2001nonuniform}, and we have the inclusion $V_{\beta}^p(\varphi) \subseteq V_{\beta}^q(\varphi) \subseteq L^q(\R), \text{ for any } 1\le p < q \le \infty$ \cite[Corollary 2.5]{aldroubi2001nonuniform}.

%\begin{lemma}\cite[Theorem 9.2.5]{christensen2003introduction}
  %  Let $\phi\in L^2(\R)$ and $\beta>0$. Assume that there exist constants $0<A\le B< \infty$ such that
   % \begin{align}
   %     \beta A \le \Psi_{\beta}(t) \le \beta B
    %\end{align}
    %for almost every $t\in [0,\frac{1}{\beta}]$. Then $V_{\beta}^2(\phi)$ is a closed subspace of $L^2(\R)$ and $\{\phi(\cdot - \beta n): n\in \Z \}$ is a Riesz basis for $V_{\beta}^2(\phi)$.
%\end{lemma}
      
    As \eqref{eq90} holds,
      there exists a unique dual generator $\widetilde{\varphi} \in V_{\beta}^2(\varphi)$ such that $\{\widetilde{\varphi}(\cdot - \beta n): n\in \Z \}$ is a Riesz basis for $V_{\beta}^2(\varphi)$, and  $\varphi$ and $\widetilde{\varphi}$ satisfy the biorthogonality relation
    \begin{align*}
        \langle \widetilde{\varphi}(\cdot-\beta n), \varphi(\cdot - \beta k) \rangle = \delta_{n, k},
    \end{align*}
    where $\delta_{n,k} = 1$ for $n=k $ and $\delta_{n,k} = 0$ for $n\ne k$. For $x\in\mathbb R$, we denote by $T_x$ the translation operator
$$T_x f(t):=f(t-x), \ \ t\in\mathbb R.$$ 
Moreover, $f\in V_{\beta}^2 (\varphi)$ admits the following biorthogonal expansion 
    \begin{align*}
        f = \sum_{n\in \Z} \langle f,T_{\beta n} \varphi \rangle T_{\beta n} \widetilde{\varphi},
    \end{align*}
    which holds in $L^2(\R)$.
    The following proposition says that  $f\in V^\infty_\beta(\varphi) \subset L^\infty(\R)$ can still be represented in a biorthogonal expansion and the series converges uniformly on compact intervals of the real line. 

    \begin{proposition}\cite[Proposition 2.5]{grohs2024stable}\label{prop.bio.expa}  Let $\varphi=e^{-\frac{\cdot^2}{2\sigma^2}}$ be the Gaussian function  and $V_{\beta}^{\infty}(\varphi)$ be the shift-invariant space with step-size $\beta>0$ in \eqref{gauss.sis.def}.  
If $f \in V^\infty_\beta(\varphi)$,  then 
\begin{align}\label{bio.expa.inf}
    f(t) = \sum_{n\in \Z} \langle f, T_{\beta n}\varphi \rangle T_{\beta n}\widetilde{\varphi}(t), \ \ t\in \R,
\end{align}
where the series on the right converges uniformly on compact intervals of the real line.
\end{proposition}

    \subsection{Riesz basis and dual generator}

    A key fact for our analysis is that the tensor product $\varphi_\xi$ in \eqref{def.auxic}  of a Gaussian also generates a shift-invariant space.
    Specifically, $\{\varphi_\xi(\cdot - \frac{\beta}{2} n): n\in \mathbb Z\}$ forms a Riesz basis for $V_{\frac{\beta}{2}}^{2}(\varphi_\xi)$, cf. Proposition \ref{prop3}. We now characterize the Riesz basis  and derive an explicit expression for the dual generator $\widetilde{\varphi_\xi}$, which will be used in the reconstruction formula. We first recall the Jacobi theta function.

    Recall that the Jacobi theta function of third kind $\vartheta_3:\C\times (0,1) \to \C$ is defined by
	\begin{align}\label{eq.theta}
		\vartheta_3(z,c):= \sum_{n\in \Z} c^{n^2} e^{2niz}.
	\end{align}
	For every fixed $c\in (0,1)$, $\vartheta_3$ is periodic in $z$ with period $\pi$. Fix $c\in (0,1)$, it is well known that $\vartheta_3(z,c)$ is minimal at point $z= \frac{\pi}{2}$ and maximal at point $z= 0$ (see \cite[p.~178]{janssen1996some}).

 To give a characterization of Riesz basis and the dual generator, we first recall the definition of $\varphi_\xi$ in \eqref{def.auxic}.  The following characterization has been established  in \cite{grohs2024stable}. For the sake of completeness, we include the proof in Appendix~\ref{proof.prop3}.

    \begin{proposition}\label{prop3}
      Let $\varphi=e^{-\frac{\cdot^2}{2\sigma^2}}$ be the Gaussian function  and $V_{\beta}^{\infty}(\varphi)$ be the shift-invariant space with step-size $\beta>0$ in \eqref{gauss.sis.def}. Then  $(\varphi(\cdot - \beta n))_{n\in \Z}$ forms a Riesz basis for $V_{\beta}^{2}(\varphi)$. 
Moreover, for every $\xi\in \R$, $(\varphi_\xi(\cdot - \frac{\beta}{2} n))_{n\in \Z}$ forms a Riesz basis for $V_{\beta/2}^{2}(\varphi_\xi)$. The dual function of $\varphi_\xi$ is  given by
\begin{equation}\label{dual.phi_w}
\widetilde{\varphi_\xi} = \sqrt{2} e^{\frac{\xi^2}{4\sigma^2}} T_{\frac{\xi}{2}} \mathcal{F}^{-1} \Lambda,
\end{equation} 
where
        \begin{align}\label{lambda.eq}
		\Lambda(t) = \frac{e^{-\pi^2 \sigma^2 t^2}}{\vartheta_3(\frac{\beta}{2}\pi t, e^{-\frac{\beta^2}{8\sigma^2}})}.
	\end{align}
    \end{proposition}

  \medskip

\section{Uniqueness of STLCT phase retrieval in Gaussian shift-invariant spaces}\label{uniq.sec}

\medskip 

Let $V_{\beta}^{\infty}(\varphi)$ be the Gaussian shift-invariant space defined in \eqref{gauss.sis.def}. For complex-valued signals in $V_{\beta}^{\infty}(\varphi)$, phaseless pointwise sampling does not in general yield phase retrieval because of the conjugate ambiguity \cite{chen2024conjugate,chen2022phase, grochenig2020phase}. In contrast, phaseless Gabor measurements provide a natural alternative. Grohs and Liehr \cite{grohs2023injectivity} proved that any $f \in V_{\beta}^{\infty}(\varphi)$ is uniquely determined, up to a global phase, by its phaseless Gabor measurements $|\mathcal{V}_{\varphi} f(\frac{\beta}{2+\epsilon}\Z, \Z)|$ for any $\epsilon > 0$. Subsequently, they considered the semi-discrete grid $\frac{\beta}{2}\Z \times \R$ and further derived an explicit reconstruction formula and quantitative local stability estimates \cite{grohs2024stable}.  Since the short-time linear canonical transform (STLCT) in \eqref{stlct.def} generalizes the Gabor transform and the parameter matrix $\mathbf{A}$ provides additional flexibility, it is natural to ask whether the corresponding uniqueness and stability results for phase retrieval continue to hold for phaseless STLCT measurements.

More precisely, given $f\in V_{\beta}^{\infty}(\varphi)$, we investigate  whether $f$ can be recovered, up to a global unimodular constant, from its phaseless STLCT measurements $\bigl|S_{\check{\varphi}}^{\mathbf A}f(\tfrac{\beta}{2}\mathbb Z,\mathbb R)\bigr|$,  where $\check{\varphi}$ is the chirp-modulated Gaussian introduced in \eqref{def.auxia}. Start from showing that  the auxiliary family $
f_\xi:=(T_\xi f)\overline f$ lives in the shift-invariant space $V_{\beta/2}^{\infty}(\varphi_\xi)$ for every $\xi\in\mathbb R$ in Lemma \ref{auxi.lem}. Hence $f_\xi$ admits a biorthogonal expansion as in \eqref{bio.expa.inf} with respect to the translates of the dual generator $\widetilde{\varphi_\xi}$, and the coefficients in the biorthogonal expansion are determined by the phaseless STLCT measurements, see Lemma \ref{lem4}. This allows us to recover the family $\{f_\xi \}_{\xi \in\mathbb R}$ and consequently reconstruct $f$ up to a global unimodular constant. 

The main result of this section is Theorem \ref{pr.gassian.thm1}, in which we prove that every function $f$  in the Gaussian shift-invariant space $V_{\beta}^{\infty}(\varphi)$ is uniquely determined, up to a global unimodular constant, by its phaseless STLCT measurements $
\bigl|S_{\check{\varphi}}^{\mathbf A}f(\tfrac{\beta}{2}\mathbb Z,\mathbb R)\bigr|$. Moreover, Theorem \ref{pr.gassian.thm1} provides an explicit reconstruction formula for $f$.

\smallskip

We begin with the fact that the auxiliary family $f_\xi$ remains in a Gaussian shift-invariant space.

\begin{lemma}\cite[Proposition 3.3]{grohs2024stable}\label{auxi.lem}
	Let $\varphi=e^{-\frac{\cdot^2}{2\sigma^2}}$  be the Gaussian function and let $V_{\beta}^{\infty}(\varphi)$ be the shift-invariant space with step-size $\beta>0$ in \eqref{gauss.sis.def}. If $f\in V_{\beta}^{\infty}(\varphi)$ and $f_\xi, \varphi_\xi$ are defined by  \eqref{def.auxic}, then
	$$f_\xi \in V_{\beta/2}^{\infty}(\varphi_\xi)$$
	for every $\xi\in \R$. 
\end{lemma}

Thus, for each fixed $\xi\in\R$, the function $f_\xi$ belongs to a Gaussian shift-invariant space with step-size $\beta/2$. The next lemma shows that the coefficients arising from the biorthogonal  representation are encoded in the phaseless STLCT measurements.

\begin{lemma}\label{lem4}
	Let $f\in L^{\infty}(\R)$ and $\omega\in L^2(\R)$. Then for any $x, \xi\in \R$,
	\begin{equation}\label{eq1}
		\langle f_\xi, T_x \omega_\xi \rangle
		= e^{\frac{2\pi iax\xi}{b}} (\mathcal{F} |S_{\check{\omega}}^{\mathbf{A}} f(x,\cdot)|^2)\Bigl(\frac{\xi}{b}\Bigr).
	\end{equation}
\end{lemma}

The proof of Lemma \ref{lem4} will be given in Appendix \ref{proof.lem4}.

\medskip

We are now in a position to state the main uniqueness result.  Lemmas \ref{auxi.lem} and \ref{lem4} provide the two ingredients needed for phase retrieval from phaseless STLCT measurements.  Lemma \ref{auxi.lem} shows that  the auxiliary function $f_\xi$ belongs to the Gaussian shift-invariant space $V_{\beta/2}^{\infty}(\varphi_\xi)$, and hence admits a biorthogonal expansion as in \eqref{bio.expa.inf}. Lemma \ref{lem4} shows that the coefficients in this expansion are encoded in the phaseless STLCT measurements. Combining these facts yields the following theorem.

\begin{theorem}\label{pr.gassian.thm1}
	Let $f$ be the function in the Gaussian shift-invariant space $V_{\beta}^{\infty}(\varphi)$ in \eqref{gauss.sis.def} and  $p\in \R$ be such that $f(p)\ne 0$. Then $f$ is uniquely determined, up to a global unimodular constant, from $
	\big|S_{\check{\varphi}}^{\mathbf{A}} f(\tfrac{\beta }{2} \Z,\R)\big|$.  
	Moreover, there exists a unimodular constant $\tau\in \T$ such that
	\begin{align*}
		f(p+\xi)
		=
		\tau |f(p)|^{-1}
		\sum_{n\in \Z}
		e^{-\frac{\pi ia\beta n\xi}{b}}
		\left(\int_{\R} |S_{\check{\varphi}}^{\mathbf{A}} f(\tfrac{\beta  n}{2},t)|^2 e^{\frac{2\pi i \xi t}{b}} dt \right)
		T_{\frac{\beta}{2}n} \widetilde{\varphi_\xi }(p+\xi), \ \xi\in \R,
	\end{align*}
	where
	\begin{equation}\label{abs.f.ed}
		|f(p)|
		=
		\left(
		\sum_{n\in \Z}
		\left(\int_{\R} |S_{\check{\varphi}}^{\mathbf{A}} f(\tfrac{\beta  n}{2},t)|^2 dt \right)
		T_{\frac{\beta}{2}n} \widetilde{\varphi}_{\xi=0}(p)
		\right)^{\frac{1}{2}}.
	\end{equation}
\end{theorem}

\begin{proof}
	Since $f(p)\neq 0$, we may write $
	\tau:=\frac{f(p)}{|f(p)|}\in\T$. 
	Then, by the definition of $f_\xi$ in \eqref{def.auxic}, we have
	\begin{equation}\label{uni.thm.eq1}
		f(p+\xi)=\tau |f(p)|^{-1}\overline{f_\xi(p+\xi)}, \  \xi\in\R.
	\end{equation}

	By Lemma \ref{auxi.lem} and Proposition \ref{prop.bio.expa}, the function $f_\xi \in V_{\beta/2}^{\infty}(\varphi_\xi)$ admits the biorthogonal expansion
	\begin{equation}\label{uni.thm.eq2}
		f_\xi
		=
		\sum_{n\in \Z}
		\langle f_\xi , T_{\frac{\beta}{2} n} \varphi_\xi \rangle
		T_{\frac{\beta}{2} n} \widetilde{\varphi_\xi}
	\end{equation}
	 for every $\xi\in\R$, where the series converges uniformly on compact intervals of $\R$.

	Next, Lemma~\ref{lem4} with $\omega=\varphi$ and $x=\frac{\beta}{2}n$ gives
	\begin{equation}\label{uni.thm.eq3}
		\langle f_\xi , T_{\frac{\beta}{2} n} \varphi_\xi \rangle
		= e^{\frac{\pi ia\beta n\xi}{b}}
		\int_{\R} |S_{\check{\varphi}}^{\mathbf{A}} f(\tfrac{\beta  n}{2},t)|^2 e^{-\frac{2\pi i \xi t}{b}} dt, \  n\in\Z.
	\end{equation}
	Substituting \eqref{uni.thm.eq3} into \eqref{uni.thm.eq2}, evaluating at $t=p+\xi$, and combining with \eqref{uni.thm.eq1}, we obtain
	\begin{align*}
		f(p+\xi) = 	\tau |f(p)|^{-1}
		\sum_{n\in \Z}
		e^{-\frac{\pi ia\beta n\xi}{b}}
		\left(\int_{\R} |S_{\check{\varphi}}^{\mathbf{A}} f(\tfrac{\beta  n}{2},t)|^2 e^{\frac{2\pi i\xi t}{b}} dt \right)
		T_{\frac{\beta}{2}n} \widetilde{\varphi_\xi }(p+\xi), \  \xi\in \R.
	\end{align*}

	It remains to determine $|f(p)|$ from the measurements. Setting $\xi=0$ in \eqref{uni.thm.eq2} gives 
		 $$ f_0(p) =
		\sum_{n\in\Z}
		\langle f_0,T_{\frac{\beta}{2}n}\varphi_{\xi=0}\rangle\,
		T_{\frac{\beta}{2}n}\widetilde{\varphi}_{\xi=0}(p).$$
	Since $f_0=|f|^2$, we have $f_0(p)=|f(p)|^2$. Moreover, by \eqref{uni.thm.eq3} with $\xi=0$,
 $$ \langle f_0,T_{\frac{\beta}{2}n}\varphi_{\xi=0}\rangle = 	\int_{\R}|S_{\check{\varphi}}^{\mathbf{A}} f(\tfrac{\beta}{2}n,t)|^2\,dt.$$
	Therefore,
	$$ |f(p)|^2 = 	\sum_{n\in \Z}\left(\int_{\R} |S_{\check{\varphi}}^{\mathbf{A}} f(\tfrac{\beta  n}{2},t)|^2 dt \right)T_{\frac{\beta}{2}n} \widetilde{\varphi}_{\xi=0}(p),$$
	and hence
\begin{equation*}
		|f(p)| =\left(
		\sum_{n\in \Z}
		\left(\int_{\R} |S_{\check{\varphi}}^{\mathbf{A}} f(\tfrac{\beta  n}{2},t)|^2 dt \right)
		T_{\frac{\beta}{2}n} \widetilde{\varphi}_{\xi=0}(p)
		\right)^{\frac{1}{2}}.
\end{equation*}
	This proves the reconstruction formula and shows that $f$ is uniquely determined, up to a unimodular constant, by $
	|S_{\check{\varphi}}^{\mathbf{A}} f(\tfrac{\beta }{2} \Z,\R)|$.
	\qedhere
\end{proof}

We note that $|f(p)|$ does not need to be prescribed in advance, since it is determined by the phaseless STLCT measurements through \eqref{abs.f.ed}.

\smallskip

For $$
\mathbf{A}=\begin{pmatrix} 0 & 1 \\ -1 & 0 \end{pmatrix},$$ 
the STLCT becomes the Gabor transform $\mathcal{V}_{\varphi}$, and Theorem \ref{pr.gassian.thm1} yields  \cite[Theorem 3.6]{grohs2024stable}.

\begin{corollary}[{\cite[Theorem 3.6]{grohs2024stable}}]\label{uni.gabor}
	Let $\varphi=e^{-\frac{\cdot^2}{2\sigma^2}}$  be the Gaussian function and let $V_{\beta}^{\infty}(\varphi)$ be the shift-invariant space with step-size $\beta>0$ in \eqref{gauss.sis.def}.  Then $f\in V_{\beta}^{\infty}(\varphi)$ is uniquely determined, up to a unimodular constant, by $|\mathcal{V}_\varphi f(\frac{\beta }{2} \Z,\R)|$. Moreover, there exists a unimodular constant $\tau\in \T$ such that
	\begin{align*}
		f(p+\xi)
		=
		\tau |f(p)|^{-1}
		\sum_{n\in \Z}
		\left(\int_{\R} |\mathcal{V}_\varphi f(\tfrac{\beta  n}{2},t)|^2 e^{2\pi i \xi t} dt \right)
		T_{\frac{\beta}{2}n} \widetilde{\varphi_\xi}(p+\xi), \  \xi\in \R,
	\end{align*}
	where
	\begin{equation*}
		|f(p)|
		=
		\left(
		\sum_{n\in \Z}
		\left(\int_{\R} |\mathcal{V}_\varphi f(\tfrac{\beta  n}{2},t)|^2 dt \right)
		T_{\frac{\beta}{2}n} \widetilde{\varphi}_{\xi=0}(p)
		\right)^{\frac{1}{2}}.
	\end{equation*}
\end{corollary}

\medskip
 
\section{Stability of STLCT phase retrieval in Gaussian shift-invariant spaces}\label{stable.sec}

\medskip 

Having established the uniqueness of STLCT phase retrieval in Gaussian shift-invariant space, we now turn to the question of stability. It is well known that phase retrieval cannot be uniformly stable in infinite-dimensional spaces \cite{alaifari2017phase,cahill2016phase}.  Accordingly, we work with a conditional notion of stability. Given an interval $I\subset\R$, we ask whether there exists a constant $\widetilde C>0$, depending on $f, g\in V_\beta^\infty(\varphi)$, such that
\begin{equation}\label{stable.eq.general}
\min_{\tau\in\mathbb T}\|f-\tau g\|_{L^\infty(I)}
\le \widetilde C
\bigl\||S_{\check{\varphi}}^{\mathbf A}f|^2-|S_{\check{\varphi}}^{\mathbf A}g|^2\bigr\|_{\mathcal D}%{\frac{\beta}{2},\infty}
\end{equation}
holds for all $g\in V_\beta^\infty(\varphi)$, where $\|\cdot\|_{\mathcal D}$ is a suitable norm on the measurement space. The main issue is then to understand how the stability constant depends on the geometric size of the interval and on lower bounds for the signal at suitable reference points.

In this section, we first establish such a stability estimate on a bounded interval, see Theorem \ref{local.stable.thm1}. More precisely, we show that the reconstruction of signals in the Gaussian shift-invariant space $V_\beta^\infty(\varphi)$ from phaseless STLCT measurements is stable  on an interval $I=[p-r,p+r]$. However, this local estimate also reveals an essential difficulty, that is  the corresponding stability constant contains the exponential factor $e^{\frac{r^2}{4\sigma^2}}$ depending on the radius $r$ of the interval. Consequently, if one applies the local estimate directly on a large interval, then the stability constant deteriorates exponentially with the interval length. This shows that a purely local argument is not sufficient for obtaining a meaningful stability result on large intervals.

To overcome this challenge, we adopt the following anchor point condition, which was introduced in the Gabor phase retrieval setting in \cite{grohs2024stable}.

\smallskip 
\noindent\textbf{Condition $(\mathbf P)$.} There exist $J\ge 2$ points
$$ p_1<p_2<\cdots<p_J$$
and a constant $\gamma>0$ such that
$$ |f(p_j)|\ge \gamma, \ 1\le j\le J.$$

Condition $(\mathbf P)$ provides a family of anchor points at which the signal is bounded away from zero. This makes it possible to replace a single stability estimate on a large interval by a phase propagation argument along adjacent intervals centered at the anchor points.  
The local stability estimate in Theorem~\ref{local.stable.thm1} is then applied only at the scale of the maximal distance between consecutive anchors. In Theorem~\ref{stable.thm2}, we show that the exponential factor in the stability bound is of the form $e^{\frac{r^2}{4\sigma^2}}$, where $
r:=\max_{1\le j\le J-1}(p_{j+1}-p_j)$, rather than depending on the total length of the interval. The passage from local to global stability therefore introduces only a linear accumulation with respect to the number of propagation steps, and in particular avoids exponential growth with the interval length.

The stability bound in \eqref{stable.eq.general} also depends on the constant
\begin{equation}\label{Csigma.def}
    C(\sigma,\beta):=  \sup_{t\in\R}\sum_{n\in\Z}     \left|T_{\frac{\beta}{2}n}\mathcal{F}^{-1}\Lambda(t)\right|,
\end{equation}
where $\Lambda$ is defined in \eqref{lambda.eq}. Under suitable assumptions on
$\sigma$ and $\beta$, there exist constants $K=K(\sigma,\beta)$ and
$\nu=\nu(\sigma,\beta)$ such that
\begin{align}\label{eq6}
    |\mathcal{F}^{-1}\Lambda(t)|\le Ke^{-\nu|t|}, \ t\in\R.
\end{align}
For example, if $\frac{\beta}{4}\le \sigma\le \frac{\beta}{2}\le 1$, then
\eqref{eq6} holds with $K\le \frac{205}{\sigma}$ and $\nu\ge \frac14$,
cf. \cite[Lemma~3.11]{grohs2024stable}. In this case, for every $t\in\R$,
\begin{align*}
    \sum_{n\in\Z}\left|T_{\frac{\beta}{2}n}\mathcal{F}^{-1}\Lambda(t)\right|
    &\le
    K\sum_{n\in\Z}e^{-\nu|t-\frac{\beta}{2}n|}  \le
    K\Bigl(2+\int_{\R}e^{-\frac{\nu\beta}{2}|x|}\,dx\Bigr)
    =
    2K\Bigl(1+\frac{2}{\nu\beta}\Bigr).
\end{align*}
Hence $C(\sigma,\beta)$ is well defined.

To measure the discrepancy of phaseless STLCT measurements, we introduce the mixed norm
$$\|F\|_{\alpha,\infty} := \sup_{n\in\mathbb Z}\|F(\alpha n,\cdot)\|_{L^1(\mathbb R)}, \ \alpha>0,$$
for measurable functions $F$ on $\alpha\mathbb Z\times\mathbb R$. If $F$ is
defined on $\mathbb R^2$, then we write
$$
\|F\|_{\alpha,\infty}:= \|F|_{\alpha\mathbb Z\times\mathbb R}\|_{\alpha,\infty},$$
where $F|_{\alpha\mathbb Z\times\mathbb R}$ denotes the restriction of $F$ to
$\alpha\mathbb Z\times\mathbb R$.

\smallskip 
We begin with a local stability estimate on a bounded interval.

\begin{theorem}\label{local.stable.thm1}
Let $f,g$ be the Gaussian shift-invariant functions in $V_\beta^\infty(\varphi)$, and let $p\in\R$ be such that
$f(p)\neq 0$ and $g(p)\neq 0$. For $r>0$, set  $ I:=[p-r,p+r].$
Then
\begin{equation}\label{local.stable.eq1}
\min_{\tau \in \T}\|f-\tau g\|_{L^\infty(I)} \le
\sqrt{2}\zeta\,C(\sigma,\beta)\bigl\||S_{\check{\varphi}}^{\mathbf A}f|^2-|S_{\check{\varphi}}^{\mathbf A}g|^2\bigr\|_{\frac{\beta}{2},\infty},
\end{equation}
where $\zeta := \frac{1}{|g(p)|} \left( \frac{\|f\|_{L^\infty(I)}}{|f(p)| + |g(p)|} + e^{\frac{r^2}{4\sigma^2}} \right)$ and $C(\sigma,\beta)$ is given in \eqref{Csigma.def}. 
\end{theorem}

\begin{proof}
By Lemma \ref{auxi.lem},  we have $f_\xi, g_\xi\in V_{\beta/2}^{\infty}(\varphi_\xi)$ for every $\xi\in\R$.
Applying the biorthogonal expansion together with Lemma \ref{lem4}, we obtain
\begin{align}\label{eq13}
    |f_\xi(t)-g_\xi(t)|  &=
    \left|   \sum_{n\in\Z}  \langle f_\xi,T_{\frac{\beta}{2}n}\varphi_\xi\rangle
    T_{\frac{\beta}{2}n}\widetilde{\varphi_\xi}(t)   -   \sum_{n\in\Z}   \langle g_\xi,T_{\frac{\beta}{2}n}\varphi_\xi\rangle T_{\frac{\beta}{2}n}\widetilde{\varphi_\xi}(t) \right| \notag\\
    &= \left|  \sum_{n\in\Z}  e^{\frac{\pi ia\beta n\xi}{b}}  \int_{\R}
    \Bigl(   |S_{\check{\varphi}}^{\mathbf A}f(\tfrac{\beta}{2}n,s)|^2 -   |S_{\check{\varphi}}^{\mathbf A}g(\tfrac{\beta}{2}n,s)|^2\Bigr) e^{-\frac{2\pi i\xi s}{b}}\,ds\, T_{\frac{\beta}{2}n}\widetilde{\varphi_\xi}(t)    \right| \notag\\
    &\le   \sum_{n\in\Z}  \Bigl\|
    |S_{\check{\varphi}}^{\mathbf A}f(\tfrac{\beta}{2}n,\cdot)|^2  -
    |S_{\check{\varphi}}^{\mathbf A}g(\tfrac{\beta}{2}n,\cdot)|^2  \Bigr\|_{L^1(\R)}  \,
   |T_{\frac{\beta}{2}n}\widetilde{\varphi_\xi}(t)| \notag\\
    &\le \sqrt{2}\,   e^{\frac{\xi^2}{4\sigma^2}}\,  C(\sigma,\beta)\,    \bigl\|
    |S_{\check{\varphi}}^{\mathbf A}f|^2  -  |S_{\check{\varphi}}^{\mathbf A}g|^2
    \bigr\|_{\frac{\beta}{2},\infty}, \ \  \xi,t\in\R.
\end{align}
Here the second identity follows from Lemma~\ref{lem4}, while the last inequality
uses Proposition~\ref{prop3} and the definition of $C(\sigma,\beta)$.

Setting $\xi=0$ in \eqref{eq13}, we obtain
\begin{equation}\label{abs.eq1}
    \bigl||f(t)|^2-|g(t)|^2\bigr|   \le  \sqrt{2}C(\sigma,\beta)  \bigl\|
    |S_{\check{\varphi}}^{\mathbf A}f|^2   -    |S_{\check{\varphi}}^{\mathbf A}g|^2
    \bigr\|_{\frac{\beta}{2},\infty},
\  t\in\R.
\end{equation}

Now define $
\tau_1:=\frac{\overline{g(p)}\,|f(p)|}{|g(p)|\,\overline{f(p)}}\in\mathbb T.$ 
For $\xi\in\R$ with $|\xi|\le r$, we have
\begin{align*}
   & |f(p+\xi)-\tau_1 g(p+\xi)| %  =   \left|  \frac{\overline{f(p)}}{|f(p)|}f(p+\xi)  -  \frac{\overline{g(p)}}{|g(p)|}g(p+\xi)   \right| 
    =   \left| \frac{1}{|f(p)|}f_{-\xi}(p)  -   \frac{1}{|g(p)|}g_{-\xi}(p)    \right| \\
    &\le  \left|    \frac{|g(p)|-|f(p)|}{|f(p)g(p)|}   \right|  |f_{-\xi}(p)|  +
    \frac{1}{|g(p)|}|f_{-\xi}(p)-g_{-\xi}(p)| \\
    &= \frac{1}{|g(p)|}  \left(   \frac{\bigl||g(p)|^2-|f(p)|^2\bigr|}{|f(p)|+|g(p)|}\,|f(p+\xi)|
    +   |f_{-\xi}(p)-g_{-\xi}(p)| \right).
\end{align*}
Using \eqref{eq13} and \eqref{abs.eq1}, together with $|\xi|\le r$, we conclude that
\begin{equation*} |f(p+\xi)-\tau_1 g(p+\xi)| \le\sqrt{2}\zeta C(\sigma,\beta)
\bigl\| |S_{\check{\varphi}}^{\mathbf A}f|^2 - |S_{\check{\varphi}}^{\mathbf A}g|^2\bigr\|_{\frac{\beta}{2},\infty},
\end{equation*} 
where
$\zeta=\frac{1}{|g(p)|}\left(\frac{\|f\|_{L^\infty(I)}}{|f(p)|+|g(p)|}+e^{\frac{r^2}{4\sigma^2}}\right)$.

Taking the supremum over $|\xi|\le r$ yields
%\begin{equation*}
%\|f-\tau_1 g\|_{L^\infty(I)} \le \sqrt{2}\zeta C(\sigma,\beta)
%\bigl\| |S_{\check{\varphi}}^{\mathbf A}f|^2- |S_{\check{\varphi}}^{\mathbf A}g|^2 \bigr\|_{\frac{\beta}{2},\infty},
%\end{equation*} 
%and therefore
\begin{equation*}
\min_{\tau \in \T}\|f-\tau g\|_{L^\infty(I)} \le
\|f-\tau_1 g\|_{L^\infty(I)}\le
 \sqrt{2}\zeta C(\sigma,\beta) \bigl\||S_{\check{\varphi}}^{\mathbf A}f|^2-
|S_{\check{\varphi}}^{\mathbf A}g|^2\bigr\|_{\frac{\beta}{2},\infty}.
\end{equation*} 
	Hence, we complete the proof.   
\end{proof}

Theorem~\ref{local.stable.thm1} shows that STLCT phase retrieval is stable on a single interval $
I=[p-r,p+r]$.  Meanwhile, the estimate \eqref{local.stable.eq1} makes clear the two sources governing the local stability constant, one is the the exponential factor $ e^{\frac{r^2}{4\sigma^2}}$  which grows with the radius $r$ of the interval, and the other is the reciprocal factors involving the anchor values $|f(p)|$ and $|g(p)|$.

For $
\mathbf{A}=
\begin{pmatrix}
0 & 1\\
-1 & 0
\end{pmatrix}$,
the STLCT reduces to the Gabor transform. In this case, Theorem \ref{local.stable.thm1} agrees with \cite[Lemma~3.8]{grohs2024stable}.

Under Condition~$(\mathbf P)$, the local stability estimate from Theorem~\ref{local.stable.thm1} can be applied at a family of anchor points where the signal is bounded away from zero. Let $ r:=\max_{1\le j\le J-1}(p_{j+1}-p_j)$ and define $ I_j:=[p_j-r,p_j+r],  1\le j\le J.$  Since the distance between two consecutive anchor points is at most $r$, these intervals overlap successively and cover the whole interval
$$ I=[p_1-r,p_J+r].$$

Theorem \ref{local.stable.thm1} provides a local stability estimate on each $I_j$.  The main issue is then to pass from these local estimates to a single stability bound on the whole interval $I$. To this end, we compare the phases at neighboring anchor points and propagate the phase information step by step from one anchor interval to the next. In this way, the global stability estimate is obtained by iterating the local estimate only across adjacent anchor pairs, rather than by applying it once on the entire interval. As a consequence, the exponential factor appearing in the final stability constant is determined by  the maximal spacing between consecutive anchor points
$$r:=\max_{1\le j\le J-1}(p_{j+1}-p_j),$$
 instead of by the total length of $I$. This is precisely the mechanism that prevents exponential deterioration of the stability constant with respect to the interval length. The resulting global stability bound is stated in the following theorem.

\begin{theorem}\label{stable.thm2}
Let $f$ be a Gaussian shift-invariant function in $V_\beta^\infty(\varphi)$, and 
$
p_1<\cdots<p_J\in\mathbb R$ 
be such that Condition $(\mathbf P)$ holds with $\gamma>0$. Set $
r:=\max_{1\le j\le J-1}(p_{j+1}-p_j)$ and $
I:=[p_1-r,p_J+r].$ 
Then, for every $g\in V_{\beta}^{\infty}(\varphi)$, one has
\begin{align}\label{stable.const}
\min_{\tau\in\mathbb T}\|f\hskip-.03in-\hskip-.03in\tau g\|_{L\hskip-.02in^\infty(\hskip-.01in I \hskip-.01in)}\hskip-.03in \le\hskip-.04in
\frac{16\sqrt{2}}{3} \hskip-.03in\left\lceil\hskip-.03in\frac{J}{2}\hskip-.03in\right\rceil \hskip-.03ine^{\frac{r^2}{4\sigma^2}}
\hskip-.03inC(\hskip-.03in\sigma,\hskip-.03in\beta)
\frac{\max\{1,\|f\|_{L\hskip-.02in^\infty(I)}\hskip-.03in+\hskip-.03in\|g\|_{L\hskip-.02in^\infty(I)}\}}{\min\{\gamma,\gamma^2\}}
\bigl\|\hskip-.02in|S_{\check{\varphi}}^{\mathbf A}f|\hskip-.01in^2\hskip-.04in-\hskip-.03in|S_{\check{\varphi}}^{\mathbf A}g|\hskip-.01in^2\hskip-.02in\bigr\|_{\frac{\beta}{2},\infty},
\end{align}
where $C(\sigma,\beta)$ is the constant defined in \eqref{Csigma.def}.
\end{theorem}

%%%% ---------

\begin{remark}[Comparison with \cite{grohs2024stable}]
    When $\mathbf{A} = \begin{pmatrix} 0 & 1 \\ -1 & 0 \end{pmatrix}$, the STLCT reduces to the Gabor transform, and Theorem~\ref{stable.thm2} recovers \cite[Theorem 3.10]{grohs2024stable}.
    Note that our stability constant in \eqref{stable.const} depends linearly on ${\rm max}\{1,\, \|f\|_{L^{\infty} (I)} + \|g\|_{L^{\infty} (I)}\}$ and the reciprocal of $\min \{\gamma, \gamma^2\}$, whereas the stability constant in \cite[Eq.(27)]{grohs2024stable} depends linearly on ${\rm max}\{\|f\|_{L^{\infty}(I)}^2,\, \|f\|_{L^{\infty} (I)} + \|g\|_{L^{\infty} (I)}\}$ and the reciprocal of $\min \{\gamma, \gamma^3\}$. Especially when $\|f\|_{L^{\infty} (I)}$ is large or $\gamma$ is small, our result substantially improves the one in \cite{grohs2024stable}.
\end{remark}

% \begin{remark}
%     For the case $\textbf{A} = 
% 	\begin{pmatrix}
% 		a & b \\
% 		c & d
% 	\end{pmatrix}
% 	$ with $a=d=0,b=-c=1$, the similar result has been shown in \cite[Theorem 3.10]{grohs2024stable}. Note that our stability constant in \eqref{stable.const} depends  linearly on ${\rm max}\{1,\|f\|_{L^{\infty}(I)},  \|f\|_{L^{\infty} (I)} + \|g\|_{L^{\infty} (I)}\}$ and  the reciprocal of $\min \{\gamma, \gamma^2\}$, whereas the stability constant in \cite[Eq.(27)]{grohs2024stable} depends linearly on ${\rm max}\{\|f\|_{L^{\infty}(I)}^2,  \|f\|_{L^{\infty} (I)} + \|g\|_{L^{\infty} (I)}\}$ and  the reciprocal of $\min \{\gamma, \gamma^3\}$. Especially when $\|f\|_{L^{\infty} (I)}$ is large or $\gamma$ is small, our result substantially improves the one in \cite{grohs2024stable}. 
% \end{remark}

For the proof of Theorem \ref{stable.thm2}, we require the following technical   lemma. 

\begin{lemma}\label{prop1}
    Let $z_1, z_2 \in \C$ with $|z_1|\ne 0$ and $|z_2|\ne 0$. Then
    \begin{align*}
        \left| \frac{z_1}{|z_1|} - \frac{z_2}{|z_2|} \right| \le 2\frac{|z_1 - z_2|}{|z_1|}.
    \end{align*}
\end{lemma}

The proof of the above lemma is given in Appendix \ref{proof.prop1}.

%\begin{lemma}\label{lem1}
%	Let $f, g \in V_{\beta}^{\infty}(\varphi), r > 0$ and $p_j, p_m \in \R$ be chosen  that
%	\begin{align*}
%		f(p_j), g(p_j), f(p_m), g(p_m) \neq 0.
%	\end{align*}
% If   $|p_m - p_j| \le r$,  then 
% \begin{align*}
% & \left| \frac{\overline{f(p_m)}}{|f(p_m)|} - \frac{\overline{f(p_j)}}%{\overline{g(p_j)}} \frac{\overline{g(p_m)}}{|g(p_m)|} \right|  \\
% \le &   \sqrt{2} C(\sigma,\beta) \| |S_{\check{\varphi}}^{\mathbf{A}} f|^2  -  |S_{\check{\varphi}}^{\mathbf{A}} g|^2 \|_{\frac{\beta}{2},\infty}  \left(  \frac{1}{|g(p_j)| (|f(p_j)|  +  |g(p_j)|)}  +  \frac{2e^{\frac{r^2}{4\sigma^2}}}{|f(p_m)g(p_j)|}  \right). 
%	\end{align*}
	%\begin{align*}
% \left| \frac{\overline{f(p_m)}}{|f(p_m)|}\hskip-.03in - \hskip-.03in\frac{\overline{f(p_j)}}{\overline{g(p_j)}} \frac{\overline{g(p_m)}}{|g(p_m)|} \right|  \hskip-.03in  \le \hskip-.03in  \sqrt{2} C(\sigma,\beta) \| |S_{\check{\varphi}}f|^2 \hskip-.03in - \hskip-.03in |S_{\check{\varphi}}g|^2 \|_{\frac{\beta}{2},\infty} \hskip-.06in \left( \hskip-.08in \frac{1}{|g(p_j)| (|f(p_j)| \hskip-.04in + \hskip-.04in |g(p_j)|)} \hskip-.03in + \hskip-.03in \frac{2e^{\frac{r^2}{4\sigma^2}}}{|f(p_m)g(p_j)|} \hskip-.08in \right). 
%	\end{align*}
%    where
%    \begin{align*}
%		B = \sqrt{2}b C(\sigma,\beta) \| |S_{\check{\varphi}}f|^2 - |S_{\check{\varphi}}g|^2 \|_{\frac{\beta}{2},\infty}. 
%    \end{align*}
%\end{lemma}

%The proof of the above theorem is given in Section \ref{proof.lem1}.

\smallskip

We now proceed to prove Theorem \ref{stable.thm2}, which gives the phase retrieval stability on the interval $I=\bigcup_{j=1}^J I_j $ under Condition $(\mathbf P)$.

\smallskip

\begin{proof}[Proof of Theorem~\ref{stable.thm2}]
Set
$\Delta := \bigl\| |S_{\check{\varphi}}^{\mathbf A}f|^2 - |S_{\check{\varphi}}^{\mathbf A}g|^2\bigr\|_{\frac{\beta}{2},\infty} $ 
and
 $B:=\sqrt{2}C(\sigma,\beta)\Delta.$ 

By \eqref{abs.eq1}, we have
\begin{equation}\label{stable.thm.proof.eq0}
\bigl\||f|^2-|g|^2\bigr\|_{L^\infty(\R)} \le B.
\end{equation}

We will do the proof by cases.
 
\smallskip
\noindent
\textbf{Case 1.} Assume that $\Delta\le \frac{\gamma^2}{2\sqrt{2}C(\sigma,\beta)}$, or equivalently,
$ B\le \frac{\gamma^2}{2}$. 

Then, for every $j\in\{1,\dots,J\}$,
$$\bigl||f(p_j)|-|g(p_j)|\bigr| \le \frac{\bigl||f(p_j)|^2-|g(p_j)|^2\bigr|}{|f(p_j)|} \le
\frac{B}{|f(p_j)|} \le \frac{\gamma}{2},$$
where we used \eqref{stable.thm.proof.eq0} and Condition~$(\mathbf P)$. Since
$|f(p_j)|\ge \gamma$, it follows that
$$ |g(p_j)|\ge \frac{\gamma}{2}, \ 1\le j\le J. $$

For $1\le j\le J$, define $ I_j:=[p_j-r,p_j+r].$ Since $|p_{j+1}-p_j|\le r$ for all $j$, these intervals cover $ I=[p_1-r,p_J+r]=\bigcup_{j=1}^J I_j.$

Let
$$ s:=\left\lceil \frac{J}{2}\right\rceil
 \ \ \text{and} \ \
\tau_s:=\frac{\overline{g(p_s)}\,|f(p_s)|}{|g(p_s)|\,\overline{f(p_s)}}\in\mathbb T. $$
Then
$$ \|f-\tau_s g\|_{L^\infty(I)} = \max_{1\le j\le J}\|f-\tau_s g\|_{L^\infty(I_j)}.$$

We now estimate $\|f-\tau_s g\|_{L^\infty(I_j)}$ for each $j$.

\medskip
\noindent
\emph{Step 1: The middle interval \(I_j\) with $j=s$.} 
By Theorem~\ref{local.stable.thm1}, we have 
 \begin{align}\label{stable.thm.eq.middle}
        \| f - \tau_s g \|_{L^{\infty} (I_j)} \le \frac{B}{|g(p_j)|} \left( e^{\frac{r^2}{4\sigma^2}} + \frac{\|f\|_{L^\infty(I_j)}}{|f(p_j)| + |g(p_j)|} \right)  \le \frac{2B}{\gamma} \left( e^{\frac{r^2}{4\sigma^2}} + \frac{2 \|f\|_{L^\infty(I_j)}}{3\gamma} \right), 
    \end{align}
since $|f(p_j)|\ge \gamma$ and $|g(p_j)|\ge \gamma/2$.

\medskip
\noindent
\emph{Step 2: The intervals \(I_j\) with \(j<s\).}
Fix $j<s$ and $\xi\in\R$ with $|\xi|\le r$. Then
 \begin{align}\label{stable.thm.eq.left1}
  &    |f(p_j + \xi) - \tau_s g(p_j + \xi)|% = \left|f(p_j + \xi) - \frac{\overline{g(p_s)} |f(p_s)|}{|g(p_s)| \overline{f(p_s)}} g(p_j + \xi) \right| \notag\\
       =   \left|\frac{\overline{f(p_s)}}{|f(p_s)|} f(p_j + \xi) \hskip-.03in - \hskip-.03in \frac{\overline{g(p_s)}}{|g(p_s)|} g(p_j + \xi)\right| \hskip-.05in \notag\\
%= \hskip-.05in \left|\frac{\overline{f(p_s)} }{|f(p_s)| } f(p_j + \xi) \hskip-.03in - \hskip-.03in \frac{\overline{g(p_s)} g(p_j) \overline{g(p_j)}}{|g(p_s)| |g(p_j)|^2} g(p_j + \xi) \right| \notag\\
     &    \le  \left| \frac{\overline{f(p_s)} }{|f(p_s)|} -   \frac{\overline{f(p_j)}}{\overline{g(p_j)}} \frac{\overline{g(p_s)}}{|g(p_s)|} \right|  |f(p_j  + \xi)|   %+ \hskip-.04in \left|\frac{\overline{g(p_s)} g(p_j) \overline{f(p_j)}}{|g(p_s)| |g(p_j)|^2} f(p_j \hskip-.04in + \hskip-.04in \xi) \hskip-.04in - \hskip-.04in \frac{\overline{g(p_s)} g(p_j) \overline{g(p_j)}}{|g(p_s)| |g(p_j)|^2} g(p_j \hskip-.04in + \hskip-.04in \xi) \right| \notag\\
    %   \le {}&  \left| \frac{\overline{f(p_s)}}{|f(p_s)|} - \frac{\overline{f(p_j)}}{\overline{g(p_j)}} \frac{\overline{g(p_s)}}{|g(p_s)|} \right| \|f\|_{L^{\infty}(I_j)}  
     + \left| \frac{\overline{g(p_s)} g(p_j)}{|g(p_s)| |g(p_j)|^2} \right| |f_ {-\xi}(p_j) - g_{-\xi}(p_j )| 
      \notag  \\
 &     \le  \left| \frac{\overline{f(p_s)}}{|f(p_s)|} - \frac{\overline{f(p_j)}}{\overline{g(p_j)}} \frac{\overline{g(p_s)}}{|g(p_s)|} \right| \|f\|_{L^{\infty}(I_j)} +\frac{2}{\gamma} e^{\frac{r^2}{4\sigma^2}}B, 
%        {}& \le \left| \frac{\overline{f(p_s)}}{|f(p_s)|} - \frac{\overline{f(p_j)}}{\overline{g(p_j)}} \frac{\overline{g(p_s)}}{|g(p_s)|} \right| \|f\|_{L^{\infty}(I_j)} + \frac{1}{|g(p_j)|} e^{\frac{r^2}{4\sigma^2}}B (\text{ by (\ref{eq13})}).
    \end{align}
where the last inequality follows from \eqref{eq13} and 
$|g(p_j)|\ge \gamma/2$.

It remains to estimate the phase mismatch
$$\delta_{j,s} := \left| \frac{\overline{f(p_s)}}{|f(p_s)|} - \frac{\overline{f(p_j)}}{\overline{g(p_j)}} \frac{\overline{g(p_s)}}{|g(p_s)|} \right|.$$
We propagate the phase from $p_s$ to $p_j$ along adjacent anchor points. Inserting
$p_{s-1}$ and applying the triangle inequality, we obtain
\begin{align}\label{phase.pro.proof.eq1}
    \delta_{j,s} &\le
     \left|  \frac{\overline{f(p_s)} f(p_{s-1}) }{|f(p_s) f(p_{s-1})|} \frac{\overline{f(p_{s-1})}}{|f(p_{s-1})|} - \frac{\overline{g(p_s)} g(p_{s-1}) }{|g(p_s) g(p_{s-1})|}  \frac{\overline{f(p_{s-1})}}{|f(p_{s-1})|}  \right|  \notag \\
     &+   \left|  \frac{\overline{g(p_s)} g(p_{s-1}) }{|g(p_s) g(p_{s-1})|} \frac{\overline{f(p_{s-1})}}{|f(p_{s-1})|}   -   \frac{\overline{g(p_s)} g(p_{s-1}) }{|g(p_s) g(p_{s-1})|} \frac{\overline{f(p_j)}}{\overline{g(p_j)}} \frac{\overline{g(p_{s-1})}}{|g(p_{s-1})|} \right| \notag\\
     &\le \left| \frac{\overline{f(p_s)}f(p_{s-1})}{|f(p_s)f(p_{s-1})|}
- \frac{\overline{g(p_s)}g(p_{s-1})}{|g(p_s)g(p_{s-1})|} \right|
+ \left| \frac{\overline{f(p_{s-1})}}{|f(p_{s-1})|}
- \frac{\overline{f(p_j)}}{\overline{g(p_j)}}\frac{\overline{g(p_{s-1})}}{|g(p_{s-1})|}\right|.
\end{align}
By Lemma~\ref{prop1}, we have 
\begin{align}\label{phase.pro.proof.eq2}
\left| \frac{\overline{f(p_s)}f(p_{s-1})}{|f(p_s)f(p_{s-1})|} - \frac{\overline{g(p_s)}g(p_{s-1})}{|g(p_s)g(p_{s-1})|}\right|
&\le\frac{2}{|f(p_s)f(p_{s-1})|}\left|\overline{f(p_s)}f(p_{s-1}) - \overline{g(p_s)}g(p_{s-1}) \right| \notag \\
\hskip-1.5in&\hskip-1.5in= \frac{2}{|f(p_s)f(p_{s-1})|} |f_{p_s-p_{s-1}}(p_s)-g_{p_s-p_{s-1}}(p_s)| \le \frac{2Be^{\frac{r^2}{4\sigma^2}}}{|f(p_s)f(p_{s-1})|},
\end{align}
where we used \eqref{eq13} and the fact that $|p_s-p_{s-1}|\le r$. 

Combine \eqref{phase.pro.proof.eq1} and \eqref{phase.pro.proof.eq2}, we have 
$$\delta_{j,s} \le \frac{2Be^{\frac{r^2}{4\sigma^2}}}{|f(p_s)f(p_{s-1})|} +
\delta_{j,s-1}.$$ 
Iterating  along the chain  $p_s,p_{s-1},\dots,p_j$, 
we obtain
\begin{align}\label{stable.thm.eq.left2}
\delta_{j,s}\le \sum_{k=j+1}^s \frac{2Be^{\frac{r^2}{4\sigma^2}}}{|f(p_k)f(p_{k-1})|} +\delta_{j ,j}&\le (s-j)\frac{2Be^{\frac{r^2}{4\sigma^2}}}{\gamma^2}   
+\frac{\bigl||f(p_j)|^2-|g(p_j)|^2\bigr|}{|g(p_j)|\bigl(|f(p_j)|+|g(p_j)|\bigr)} \notag\\
& \le (s-j)\frac{2Be^{\frac{r^2}{4\sigma^2}}}{\gamma^2}
+ \frac{4B}{3\gamma^2}.
\end{align}
Substituting \eqref{stable.thm.eq.left2} into \eqref{stable.thm.eq.left1} yields
\begin{align*}
|f(p_j+\xi)-\tau_s g(p_j+\xi)| &\le \left( (s-j)\frac{2Be^{\frac{r^2}{4\sigma^2}}}{\gamma^2} +\frac{4B}{3\gamma^2}\right)\|f\|_{L^\infty(I_j)}
+\frac{2Be^{\frac{r^2}{4\sigma^2}}}{\gamma}. 
\end{align*}
Taking the supremum over $|\xi|\le r$ gives
\begin{align}\label{stable.thm.eq.left3}
    \|f-\tau_s g\|_{L^\infty(I_j)} \le \left( \frac{2sBe^{\frac{r^2}{4\sigma^2}}}{\gamma^2}
+ \frac{4B}{3\gamma^2} \right) \|f\|_{L^\infty(I_j)} + \frac{2Be^{\frac{r^2}{4\sigma^2}}}{\gamma}, \  j<s.
\end{align}

\medskip
\noindent
\emph{Step 3: The intervals \(I_j\) with \(j>s\).}
Following the same argument in the \emph{Step 2}, propagating the phase from $p_s$ to the right, we obtain
\begin{align}\label{stable.thm.eq.right}
\|f-\tau_s g\|_{L^\infty(I_j)} \le \left( (j-s)\frac{2Be^{\frac{r^2}{4\sigma^2}}}{\gamma^2}
+ \frac{4B}{3\gamma^2} \right) \|f\|_{L^\infty(I_j)} + \frac{2Be^{\frac{r^2}{4\sigma^2}}}{\gamma}, \  j>s.
\end{align}

Combining \eqref{stable.thm.eq.middle}, \eqref{stable.thm.eq.left3}, and
\eqref{stable.thm.eq.right}, and using $|j-s|\le s$, we arrive at
\begin{equation*}
\|f-\tau_s g\|_{L^\infty(I_j)} \le \left( \frac{2sBe^{\frac{r^2}{4\sigma^2}}}{\gamma^2}
+ \frac{4B}{3\gamma^2} \right)\|f\|_{L^\infty(I_j)} + \frac{2Be^{\frac{r^2}{4\sigma^2}}}{\gamma},
\ 1\le j\le J.
\end{equation*}
Hence,
\begin{equation}\label{stable.thm.eq.case1}
\min_{\tau\in\mathbb T}\|f-\tau g\|_{L^\infty(I)} \le \|f-\tau_s g\|_{L^\infty(I)} \le \left( \frac{2sBe^{\frac{r^2}{4\sigma^2}}}{\gamma^2} +
\frac{4B}{3\gamma^2} \right)\|f\|_{L^\infty(I)} + \frac{2Be^{\frac{r^2}{4\sigma^2}}}{\gamma}.
\end{equation}

\smallskip
\noindent
\textbf{Case 2.} Assume that $
\Delta>\frac{\gamma^2}{2\sqrt{2}\,C(\sigma,\beta)}$,
or equivalently, $
B>\frac{\gamma^2}{2}$. 
Then, for every $\tau\in\mathbb T$,
$$ \|f-\tau g\|_{L^\infty(I)} \le \|f\|_{L^\infty(I)}+\|g\|_{L^\infty(I)}
< \frac{2B}{\gamma^2} \bigl(\|f\|_{L^\infty(I)}+\|g\|_{L^\infty(I)}\bigr).$$
Therefore,
\begin{equation}\label{stable.thm.eq.case2}
\min_{\tau\in\mathbb T}\|f-\tau g\|_{L^\infty(I)} \le \frac{2B}{\gamma^2} \bigl(\|f\|_{L^\infty(I)}+\|g\|_{L^\infty(I)}\bigr).
\end{equation}

Finally, combining \eqref{stable.thm.eq.case1} and \eqref{stable.thm.eq.case2}, and using 
$ s=\left\lceil \frac{J}{2}\right\rceil$, 
we obtain
\begin{align*}
&\min_{\tau\in\mathbb T}\|f-\tau g\|_{L^\infty(I)} \le \frac{16}{3}se^{\frac{r^2}{4\sigma^2}}
B\frac{\max\{1,\|f\|_{L^\infty(I)}+\|g\|_{L^\infty(I)}\}}{\min\{\gamma,\gamma^2\}} \\
&=\frac{16\sqrt{2}}{3}\left\lceil \frac{J}{2}\right\rceil e^{\frac{r^2}{4\sigma^2}} C(\sigma,\beta)
\frac{\max\{1,\|f\|_{L^\infty(I)}+\|g\|_{L^\infty(I)}\}}{\min\{\gamma,\gamma^2\}}\bigl\| |S_{\check{\varphi}}^{\mathbf A}f|^2 - |S_{\check{\varphi}}^{\mathbf A}g|^2\bigr\|_{\frac{\beta}{2},\infty}.
\end{align*}
Hence, we complete the proof.
\end{proof}

We note that,  in the proof of the above theorem, the reference anchor point is chosen as the middle point $p_s$. More generally, any point in $\{p_1,\dots,p_J\}$ may be used as the initial anchor point. However, the resulting stability estimate  will in general have a larger constant,  due to the increased number of phase propagation steps.

\medskip 

\section{Robust reconstruction from noisy discrete phaseless STLCT samples for Gaussian shift-invariant signals}\label{robust.sec}

\medskip 

Having established uniqueness and stability for STLCT phase retrieval in $V_\beta^\infty(\varphi)$, we now address the problem of practical reconstruction from finite noisy data. In this section, we develop an explicit reconstruction algorithm for recovering $f\in V_\beta^\infty(\varphi)$ from noisy phaseless STLCT samples on a finite lattice,  and we derive quantitative bounds for the reconstruction.

In Theorem \ref{pr.gassian.thm1}, we know that the auxiliary function $f_\xi$ admits the representation
\begin{equation}\label{eq:continuous-fw}
f_\xi(t) = \sum_{n\in\Z} e^{\frac{\pi ia\beta n\xi}{b}}
\left( \int_{\R} M_f\!\left(\frac{\beta n}{2},u\right)e^{-\frac{2\pi i\xi u}{b}}\,du
\right) T_{\frac{\beta}{2}n}\widetilde{\varphi_\xi}(t), \  \xi\in\R,
\end{equation} 
where 
 \begin{equation}\label{modulous.def.Mf}
M_f(x,t):=\bigl|S_{\check{\varphi}}^{\mathbf A}f(x,t)\bigr|^2,  \ (x,t)\in\R^2
\end{equation}  
are the  STLCT magnitudes.  Based on a discretization of \eqref{eq:continuous-fw} and a phase propagation procedure across adjacent anchor points, we formulate an explicit reconstruction algorithm, see Algorithm \ref{alg:stlct}.

To establish robustness of the reconstruction algorithm, we first analyze the local reconstruction error. In Theorem \ref{local-discrete-error.thm}, we derive an explicit upper bound in which the total error is decomposed into truncation of the spatial sum, frequency discretization by the trapezoidal rule, and measurement noise.   A key ingredient is that the STLCT magnitude $|S_{\check{\varphi}}^{\mathbf A}f(x,\cdot)|^2$ extends to an entire function, see Lemma \ref{lem:entire-stlct-magnitude}. This analytic structure permits the use of the trapezoidal rule with exponential convergence in the frequency variable, see Lemma \ref{lem:trapezoidal}.

Finally, combining the local error estimate with the phase propagation argument, we prove a quantitative robustness result for Algorithm \ref{alg:stlct}, see Theorem \ref{alg-stability.thm}. More precisely, for any prescribed tolerance $\varepsilon>0$, we derive explicit conditions on the sampling parameters and the noise level under which the global reconstruction error is bounded by $\varepsilon$ times the conditioning factor
\[ \kappa_{f,\gamma} := \frac{\max\{1,\|f\|_{L^\infty(I)}\}}{\min\{\gamma,\gamma^2\}},\]
which reflects the signal amplitude on the reconstruction interval together with the lower bound at the anchor points.

\subsection{The Reconstruction Algorithm} 
% From Theorem \ref{pr.gassian.thm1}, we know that the auxiliary function $f_\xi$ in \eqref{def.auxic} admits the representation
% \begin{equation}\label{eq:continuous-fw}
% f_\xi(t) = \sum_{n\in\Z} e^{\frac{\pi ia\beta n\xi}{b}}
% \left( \int_{\R} M_f\!\left(\frac{\beta n}{2},u\right)e^{-\frac{2\pi i\xi u}{b}}\,du
% \right) T_{\frac{\beta}{2}n}\widetilde{\varphi_\xi}(t), \  \xi\in\R,
% \end{equation} 
% where 
%  \begin{equation}\label{modulous.def.Mf}
% M_f(x,t):=\bigl|S_{\check{\varphi}}^{\mathbf A}f(x,t)\bigr|^2,  \ (x,t)\in\R^2
% \end{equation} 
% are the  STLCT magnitudes. 
In this section, we develop an explicit reconstruction procedure. The main idea is to discretize the integral term in \eqref{eq:continuous-fw}, use the resulting approximation to reconstruct the signal locally near suitable anchor points, and then propagate the phase information across adjacent anchor intervals to obtain a global reconstruction.

% The reconstruction is based on the continuous identity
% \begin{equation}\label{eq:continuous-fw}
% f_\xi(t) = \sum_{n\in\Z} e^{\frac{\pi ia\beta n\xi}{b}}
% \left( \int_{\R} M_f\!\left(\frac{\beta n}{2},u\right)e^{-\frac{2\pi i\xi u}{b}}\,du
% \right) T_{\frac{\beta}{2}n}\widetilde{\varphi_\xi}(t), \  t,\xi\in\R,
% \end{equation}
% which follows from Theorem \ref{pr.gassian.thm1}. The idea is to discretize \eqref{eq:continuous-fw}, then recover local reconstructions near the anchor points, and finally propagate the phase information across adjacent anchor intervals.

Let 
$$X_{N,H,h}:=\frac{\beta}{2}\{-N,\dots,N\}\times h\{-H,\dots,H\}\subset\R^2$$ 
be the discrete sampling lattice, where $N, H\in \N$ and $h>0$ is the frequency sampling step.  
The noisy phaseless STLCT data are given by
\begin{equation}\label{noisy.data.def}
Y_{n,k} = M_f\!\left(\frac{\beta n}{2},hk\right)+\eta_{n,k},  \ (n,k)\in\{-N,\dots,N\}\times\{-H,\dots,H\},
\end{equation}
where $\eta=(\eta_{n,k})\in\R^{(2N+1)\times(2H+1)}$ is a noise matrix satisfying
$ \|\eta\|_\infty\le \delta$. 
Here, for a matrix $B=(B_{n,k})$, we write $ \|B\|_\infty:=\max_n\sum_k |B_{n,k}|$.

Let $f\in V_\beta^\infty(\varphi)$ satisfy Condition $(\mathbf P)$ with anchor points
$ -s=p_1<\cdots<p_J=s$   and lower bound $\gamma>0$, and set $I:=[-s,s]$.

To identify suitable anchor points from the data, we first introduce the anchor detector
\begin{equation}\label{eq:anchor-detector}
\mathcal A(t) := h\sum_{n=-N}^{N}\sum_{k=-H}^{H} Y_{n,k}\,T_{\frac{\beta}{2}n}\widetilde{\varphi_0}(t), \ t\in[-s,s], 
\end{equation}
which is the discrete analogue of the reconstruction formula for $|f(t)|^2$.

Once the anchor points
$$p_1<\cdots<p_J$$
have been selected, we evaluate $\mathcal A$ at these points and define the discrete anchor magnitudes
\begin{equation}\label{eq:Aj-def}
A_j := h\sum_{n=-N}^{N}\sum_{k=-H}^{H}Y_{n,k}\,T_{\frac{\beta}{2}n}\widetilde{\varphi_0}(p_j)
\end{equation}
for $1\le j\le J$.

Set  $$r:=\max_{1\le j\le J-1}(p_{j+1}-p_j).$$  
For each anchor point $p_j$ and each displacement $|\xi|\le r$, we define the local reconstruction function
\begin{equation}\label{eq:Rj-def}
R_j(\xi) := \frac{h}{\sqrt{A_j}} \sum_{n=-N}^{N} e^{-\frac{\pi ia\beta n\xi}{b}}
\sum_{k=-H}^{H}Y_{n,k}\,e^{\frac{2\pi i\xi hk}{b}}T_{\frac{\beta}{2}n}\widetilde{\varphi_\xi}(p_j+\xi), \  |\xi|\le r.\end{equation} 
Comparing \eqref{eq:Rj-def} with the continuous representation \eqref{eq:continuous-fw}, we see that $\sqrt{A_j}\,\overline{R_j(\xi)}$
serves as a discrete approximation of $f_\xi(p_j+\xi)$. Thus $R_j$ provides a local reconstruction of the signal near the anchor point $p_j$.

Finally, for $j=1,\dots,J-1$, we define the phase transition factors
\begin{equation}\label{eq:rhoj-def}
\rho_j := \frac{R_j(p_{j+1}-p_j)}{|R_j(p_{j+1}-p_j)|}.
\end{equation}
The role of $\rho_j$ is to approximate the relative phase between the anchor values at $p_j$ and $p_{j+1}$, and therefore allow the phase propagation along the chain of anchor points. 

Based on the quantities introduced above, we now formulate the reconstruction algorithm for recovering the signal from discrete phaseless STLCT measurements, see Algorithm \ref{alg:stlct}.

\begin{algorithm}[h!]
\SetAlgoLined
\caption{Reconstruction from noisy phaseless STLCT samples}
\label{alg:stlct}
\textbf{Input:} The noisy data $Y=(Y_{n,k})$, parameters $N,H,h$, thresholds $r,\widetilde\gamma>0$, and $s>0$.   \\

\textbf{\bf Procedures:}
\begin{itemize}
\item[] {\bf Step 1}: Compute the anchor detector
\begin{equation}\label{mathcalA.alg.eq1} 
\mathcal A(t) = h\sum_{n=-N}^{N}\sum_{k=-H}^{H} Y_{n,k}\,T_{\frac{\beta}{2}n}\widetilde{\varphi_0}(t), \ t\in[-s,s].
\end{equation}

\item[] {\bf Step 2}: Select anchor points 
$$ -s=p_1<\cdots<p_J=s $$ 
such that
$$\mathcal A(p_j)\ge \widetilde\gamma, \  \ 
p_{j+1}-p_j\le  r,$$
and, if $J\ge 3$,
$$p_{j+2}-p_j\ge  r. $$

\item[] {\bf Step 3}: For each $j=1,\dots,J$, compute $A_j$ and $R_j$ by \eqref{eq:Aj-def} and \eqref{eq:Rj-def}.

 \qquad   \    For each $j=1,\dots,J-1$, compute $\rho_j$ by \eqref{eq:rhoj-def}.

\item[] {\bf Step 4}: Set $
\mu:=\left\lceil \frac{J}{2}\right\rceil$. Define the reconstruction $\mathcal R$ by
\begin{equation}\label{rec.alg.eq}
 \mathcal R(t):=
\begin{cases}
\overline{\rho_j\cdots\rho_{\mu-1}}\,R_j(\xi),
& t=p_j+\xi\in[p_j,p_{j+1}),\quad 1\le j\le \mu-1,\\[0.4em]
R_\mu(\xi),
& t=p_\mu+\xi\in[p_\mu,p_{\mu+1}],\\[0.4em]
\rho_\mu\cdots\rho_{j-1}\,R_j(\xi),
& t=p_j+\xi\in(p_j,p_{j+1}],\quad \mu+1\le j\le J-1.
\end{cases}   
\end{equation}

\end{itemize}

\textbf{Output:} A reconstruction $\mathcal R:[-s,s]\to \C$.
\end{algorithm}

The lower bound $\mathcal A(p_j)\ge \widetilde\gamma$ in Step 2  guarantees that each anchor point carries sufficiently large amplitude, while the spacing constraints  in Step 2 ensure that the detected anchor points are neither too separated nor too dense.  The condition $
p_{j+1}-p_j\le  r$ ensures that consecutive anchor points are close enough for stable local reconstruction and phase propagation, whereas $p_{j+2}-p_j\ge  r$ yields a bound on the total number of anchor intervals, which is crucial in the proof of robustness when passing from local discretization errors to a global reconstruction bound.  Without loss of generality, we may assume that $ r\le 2 s$. Indeed, if $r>2s$, then at most two anchor points can be selected in the interval $[-s,s]$. Otherwise, if three anchor points $p_1<p_2<p_3$ are selected, then $p_3-p_1\le 2s<r$, which contradicts the spacing condition $p_{j+2}-p_j\ge r$.

Note that the exact constant $\gamma$ in Condition $(\mathbf P)$ is typically unknown in practice. The algorithm therefore works with a threshold $\widetilde\gamma$ extracted from the data. Below we show that, under suitable assumptions, $f$ satisfies condition $(\mathbf P)$ at the chosen anchor points with an effective lower bound inherited from $\widetilde\gamma$.

\subsection{Robustness of the discrete reconstruction}

We now analyze the error introduced by frequency discretization, truncation, and measurement noise. Let $ f=\sum_{n\in\Z} c_n\varphi(\cdot-\beta n)\in V_\beta^\infty(\varphi), c\in\ell^\infty(\Z)$,
and let $p\in[-s,s]$. For $|\xi|\le r$, define
\begin{equation}\label{eq:Gp-def}
\mathcal G_p(\xi) := h\sum_{n=-N}^{N} e^{\frac{\pi ia\beta n\xi}{b}}
\sum_{k=-H}^{H}Y_{n,k}\,e^{-\frac{2\pi i\xi hk}{b}}T_{\frac{\beta}{2}n}\widetilde{\varphi_\xi}(p+\xi).
\end{equation}
If $p=p_j$, then
$$ \mathcal G_{p_j}(\xi)=\sqrt{A_j}\,\overline{R_j(\xi)}.$$
Hence the error of the local discrete reconstruction is measured by
\begin{equation}\label{eq:Ej-def}
E_j(\xi):=\bigl|f_\xi(p_j+\xi)-\mathcal G_{p_j}(\xi)\bigr|.
\end{equation}

The following theorem bounds this error explicitly.

\begin{theorem}\label{local-discrete-error.thm}
Let $f=\sum_{n\in\Z} c_n\varphi(\cdot-\beta n)\in V_\beta^\infty(\varphi), c\in\ell^\infty(\Z)$,
and let $N\hskip-.02in=\hskip-.02in\left\lceil \hskip-.02in\frac{2}{\beta}\Bigl(s+\frac{r}{2}\Bigr)\hskip-.02in\right\rceil+m$ and 
$H=\left\lceil \frac{|a|\beta N}{2h}\right\rceil+q$, 
where $m,q\in\N$ and $s,r,h>0$. Let $Y=(Y_{n,k})$  be the noisy phaseless STLCT measurements   given by \eqref{noisy.data.def}, and let $\mathcal G_p$ be defined by \eqref{eq:Gp-def}. Then, for every $p\in[-s,s]$ and every $\xi\in[-r,r]$, we have
\begin{align}\label{eq:local-discrete-error}
|f_\xi(p+\xi)-\mathcal G_p(\xi)|& \le Ke^{\frac{\xi^2}{4\sigma^2}} \left( \frac{4\sqrt{\pi}\sigma}{\nu\beta}\|c\|_\infty^2\Big(1+\frac{2\sqrt{2\pi}\sigma}{\beta}\Big)^2e^{-\frac{\nu\beta}{2}m} 
+2\sqrt{2}h\Big(1+\frac{2}{\nu\beta}\Big)\|\eta\|_\infty \right.
\nonumber\\
&\left.+2\sqrt{\pi}\sigma\|c\|_\infty^2\Big(1+\frac{2\sqrt{2\pi}\sigma}{\beta}\Big)^2 \big(1+\frac{2}{\nu\beta}\big)
\Big(\frac{2e^{1+\frac{|\xi|}{\sigma}}}{e^{\frac{b}{\sigma h}}-1}
+e^{-\frac{2\pi^2\sigma^2 h^2 q^2}{b^2}}\Big)\right),
\end{align}
where $K=K(\sigma,\beta)$ and $\nu=\nu(\sigma,\beta)$ are the decay constants in \eqref{eq6}.
\end{theorem}

The error bound in  \eqref{eq:local-discrete-error} decomposes into three terms according to their sources. The first term arises from truncating the time-domain sum to a finite range. The second term reflects the influence of measurement noise. The third term is the frequency discretization error incurred by replacing the continuous integral in \eqref{eq:continuous-fw} with a trapezoidal sum. The proof of the above theorem is included in Appendix \ref{proof.local-discrete-error.thm}.

We next record the analytic ingredients underlying Theorem~\ref{local-discrete-error.thm}.

\begin{lemma}\label{lem:entire-stlct-magnitude}
Let $f=\sum_{n\in\Z} c_n\varphi(\cdot-\beta n)\in V_\beta^\infty(\varphi), c\in\ell^\infty(\Z)$. For every $(x,t)\in\R^2$, we have 
$$ |S_{\check{\varphi}}^{\mathbf A}f(x,t)|^2 = \frac{\pi\sigma^2}{b} e^{-\frac{2\pi^2\sigma^2}{b^2}(t-ax)^2}V_x(t),$$
where $V_x(t)=\sum_{\ell\in\Z}\breve r_\ell^x\,e^{\pi i\frac{\beta\ell}{b}t}$
is a trigonometric series with coefficients
\begin{equation}\label{trig.coef.lem.def.eq1}
\breve r_\ell^x:=e^{-\pi i\frac{\beta ax}{b}\ell}\sum_{n\in\Z}c_n\overline{c_{n+\ell}}e^{-\frac{(x-\beta n)^2}{4\sigma^2}}e^{-\frac{(x-\beta (n+\ell))^2}{4\sigma^2}}
\end{equation}
satisfying 
$$ |\breve r_\ell^x| \le \|c\|_\infty^2\left(1+\frac{\sigma\sqrt{2\pi}}{\beta}\right)
e^{-\frac{\beta^2\ell^2}{8\sigma^2}}.$$  
In particular, for every fixed $x\in\R$, the map 
$$t\mapsto |S_{\check{\varphi}}^{\mathbf A}f(x,t)|^2$$ extends to an entire function on $\C$. %This extension is explicitly given by the map $\C \ni z \mapsto \frac{\pi \sigma^2}{b} e^{- \frac{2 \pi^2 \sigma^2}{b^2} (z - ax)^2} V_x(z)$. 
\end{lemma}

    The proof of above lemma is given in Appendix \ref{proof.lem.entire-stlct-magnitude}.

\smallskip 

To discretize the Fourier-type integral in \eqref{eq:continuous-fw}, we apply the trapezoidal rule to the integrand
$$W_{x,\xi}(t):=M_f(x,t)e^{-\frac{2\pi i\xi t}{b}}, \  t\in\R,$$
where $M_f(x,t)=|S_{\check{\varphi}}^{\mathbf A}f(x,t)|^2.$ 
By Lemma \ref{lem:entire-stlct-magnitude}, for every fixed $x\in\R$, the map
$$t\mapsto M_f(x,t)$$
extends to an entire function on $\C$. Hence, for every fixed $x, \xi\in\R$, the function
$$ t\mapsto W_{x,\xi}(t)=M_f(x,t)e^{-\frac{2\pi i\xi t}{b}}$$
also extends to an entire function on $\C$. This analytic continuation allows us to use the trapezoidal rule and obtain exponentially accurate discretization of the integral in \eqref{eq:continuous-fw}.

For a map $W:\R\to\C$, step size $h>0$, and truncation parameter $H\in\N\cup\{\infty\}$, we define the trapezoidal approximation by
$$ I_h^H(W):=h\sum_{k=-H}^{H}W(hk).$$
Applied to $W_{x,\xi}$, this gives
\begin{equation*}
I_h^H(W_{x,\xi})=h\sum_{k=-H}^{H}M_f(x,hk)e^{-\frac{2\pi i\xi hk}{b}}.
\end{equation*}

The following result of Trefethen and Weideman shows that the trapezoidal rule converges exponentially fast for analytic functions satisfying suitable decay conditions in a horizontal strip. 

\begin{lemma}[Trefethen--Weideman {\cite[Theorem 5.1]{trefethen2014exponentially}}]\label{lem:trapezoidal}
Let $\zeta>0$. Suppose $W$ is analytic in the strip
 $ U_\zeta:=\{z\in\C:\ |\Im z|<\zeta\}$, 
satisfies $W(x)\to 0$ uniformly as $|x|\to\infty$ in $U_\zeta$, and
$$ M(W):=\sup_{t+iy\in U_\zeta}\int_{\R}|W(t+iy)|\,dt<\infty.$$ 
Then for any $h > 0$, the trapezoidal rule approximation $I_h^{\infty} (W)$ exists and satisfies
	\begin{align*}
		| I_h^{\infty} (W) - \int_{\R} W(t)dt | \le \frac{2 M(W)}{e^{2 \pi \frac{\zeta}{h}} - 1}.
	\end{align*}
\end{lemma}

The proof of Theorem \ref{local-discrete-error.thm} is obtained by combining Lemma \ref{lem:entire-stlct-magnitude}, Lemma \ref{lem:trapezoidal}, and the decay estimate \eqref{eq6}, and we place the details in Appendix \ref{proof.local-discrete-error.thm}.

As a consequence of \eqref{eq:local-discrete-error} and the definition of the local reconstruction errors $E_j$ in \eqref{eq:Ej-def}, we obtain the uniform bound
\begin{align}\label{eq9}
\max_{1\le j\le J-1} \hskip-.06in \|E_j\|_{L^\infty([0,p_{j+1}-p_j])} & \hskip-.05in \le \hskip-.05in
D_1\|c\|_\infty^2 e^{-\frac{\nu\beta}{2}m} \hskip-.05in + \hskip-.05in D_2\|c\|_\infty^2\frac{1}{e^{\frac{b}{\sigma h}} \hskip-.05in - \hskip-.05in 1} \hskip-.05in + \hskip-.05in D_3\|c\|_\infty^2 e^{-\frac{2\pi^2\sigma^2 h^2 q^2}{b^2}}
\hskip-.06in + \hskip-.05in D_4 h\|\eta\|_\infty,
\end{align}
where
\begin{align}\label{eq8}
D_1:={}& \frac{4\sqrt{\pi}\sigma}{\nu\beta} K e^{\frac{r^2}{4\sigma^2}} \left(1+\frac{2\sqrt{2\pi}\sigma}{\beta}\right)^2,
\notag\\
D_2:={}& 4\sqrt{\pi}K\sigma e^{\frac{r^2}{4\sigma^2}} \left(1+\frac{2\sqrt{2\pi}\sigma}{\beta}\right)^2
\left(1+\frac{2}{\nu\beta}\right) e^{1+\frac{r}{\sigma}},
\notag\\
D_3:={}& 2\sqrt{\pi}K\sigma e^{\frac{r^2}{4\sigma^2}} \left(1+\frac{2\sqrt{2\pi}\sigma}{\beta}\right)^2
\left(1+\frac{2}{\nu\beta}\right),
\notag\\
D_4:={}& 2\sqrt{2}K e^{\frac{r^2}{4\sigma^2}} \left(1+\frac{2}{\nu\beta}\right).
\end{align}

The estimate \eqref{eq9} provides a uniform control of the local reconstruction errors on all anchor intervals, and will be used  to  propagate the local reconstruction error across adjacent anchor intervals and thereby derive a global error estimate for the reconstruction algorithm.

\subsection{Stability of the Reconstruction} 
We now turn to the stability of Algorithm \ref{alg:stlct}. The first step is to quantify the local error produced by the discrete reconstruction procedure. We introduce a tolerance parameter $\varepsilon>0$ and require the sampling lattice and the noise level to be chosen so that the local reconstruction errors satisfy
$$ \max_{1\le j\le J-1}\|E_j\|_{L^\infty([0,p_{j+1}-p_j])}\le \varepsilon,$$
cf. \eqref{Ej-uniform-small.eq}. Thus, $\varepsilon$ measures the total local numerical error coming from the truncation of the spatial sum, the discretization of the Fourier-type integral by the trapezoidal rule, and the measurement noise. 

To pass from this local control to a global reconstruction bound, we introduce the conditioning factor
\begin{equation}\label{kappa-def.eq}
\kappa_{f,\gamma} := \frac{\max\{1,\|f\|_{L^\infty(I)}\}}{\min\{\gamma,\gamma^2\}}.
\end{equation}
The quantity $\kappa_{f,\gamma}$ reflects the stability of the phase propagation step in the algorithm. More precisely, the factor $\|f\|_{L^\infty(I)}$ accounts for the size of the signal on the reconstruction interval, while the denominator records the lower bound $|f(p_j)|\ge \gamma$ at the anchor points. In particular, small anchor values lead to a less favorable conditioning of the reconstruction problem.

The main result below shows that, under explicit assumptions on the sampling lattice and the noise level, the anchor points detected by Algorithm \ref{alg:stlct} satisfy Condition $(\mathbf P)$, the phase transition factors are well defined, and the output function $\mathcal R$ approximates the underlying signal on $I=[-s,s]$,  up to a global unimodular constant, with reconstruction error bounded by a constant multiple of $\varepsilon\kappa_{f,\gamma}$.

\begin{theorem}[Robustness of Algorithm~\ref{alg:stlct}]\label{alg-stability.thm}
Let $f=\sum_{n\in\mathbb Z}c_n\varphi(\cdot-\beta n)\in V_\beta^\infty(\varphi)$,  
and let $r,\gamma,s>0$ with $r\le 2s$. Assume that
\begin{equation}\label{local.error.def.eq1}
0<\varepsilon \le \min\left\{\gamma^2,\frac{\gamma^3}{\|f\|_{L^\infty(I)}}\right\}, \ I=[-s, s].\end{equation}  
Suppose that the noisy phaseless STLCT measurements $Y=(Y_{n,k})$ in \eqref{noisy.data.def} are taken on the sampling lattice 
$$X_{N,H,h}=\frac{\beta}{2}\{-N,\dots,N\}\times h\{-H,\dots,H\}$$
satisfying 
\begin{align}
N &\ge \left\lceil \frac{2}{\beta}\Bigl(s+\frac{r}{2}\Bigr)\right\rceil
+\frac{2}{\nu\beta}\log\!\left(\frac{16sD_1\|c\|_\infty^2}{r\varepsilon}\right), \label{cond.N}\\
\frac{1}{h} &\ge \frac{\sigma}{b}\log\!\left(\frac{16sD_2\|c\|_\infty^2}{r\varepsilon}+1\right), \label{cond.h}\\
H &\ge \left\lceil \frac{|a|\beta N}{2h}\right\rceil
+\frac{b}{\sqrt{2}\pi\sigma h}
\left(\log\!\frac{16sD_3\|c\|_\infty^2}{r\varepsilon}\right)^{1/2}, \label{cond.H}
\end{align}
and that the noise level satisfying 
\begin{equation}\label{cond.eta}
\|\eta\|_\infty \le \frac{r\varepsilon}{16shD_4},
\end{equation}
where $K=K(\sigma,\beta)$ and $\nu=\nu(\sigma,\beta)$ are as in \eqref{eq6}, and $D_1, D_2, D_3, D_4$ are as in \eqref{eq8}. Apply  Algorithm \ref{alg:stlct}  with  parameters $$
r, \ \tilde\gamma=\frac{3}{2}\gamma^2, \  s.$$  Let $$ -s=p_1<\cdots<p_J=s$$
be the anchor points  selected in Step 2 of Algorithm \ref{alg:stlct}. 
Then the following hold.
\begin{enumerate}

\item
The signal $f$ satisfies Condition $(\mathbf P)$ with anchor points
$p_1,\ldots,p_J$ and lower bound $\gamma$.

\item
The phase transition factors $\rho_j$ are well defined for
$1\le j\le J-1$. Moreover, the output $\mathcal R$ of
Algorithm \ref{alg:stlct} satisfies
\begin{equation}\label{alg.robust.thm.eq2}
\min_{\tau\in\mathbb T} \|f-\tau\mathcal R\|_{L^\infty(I)}
\le\frac{11}{4}\,\varepsilon\,\kappa_{f,\gamma},
\end{equation}
where
$$ \kappa_{f,\gamma} :=\frac{\max\{1,\|f\|_{L^\infty(I)}\}}{\min\{\gamma,\gamma^2\}}.$$
\end{enumerate}
\end{theorem}

\begin{proof}
We first consider the case $J\ge 3$. By Step 2 of Algorithm \ref{alg:stlct}, the selected anchor points satisfy
$$ p_{j+2}-p_j\ge r, \  1\le j\le J-2.$$ 
Thus the odd-indexed anchor points and the even-indexed anchor points are each $r$-separated. Since both subsequences are contained in the interval $[-s,s]$, whose length is $2s$, we obtain
$$ J-2\le \frac{4s}{r}.$$ 
Moreover, since $r\le 2s$, we also have $2\le 4s/r$. Hence
\begin{equation}\label{J-bound.eq}
\max\{J-2,2\}\le \frac{4s}{r}.
\end{equation}

Set $ m:=N-\left\lceil \frac{2}{\beta}\left(s+\frac{r}{2}\right)\right\rceil$ and  $
q:=H-\left\lceil \frac{|a|\beta N}{2h}\right\rceil.$ 
Let $\Xi_1:= D_1\|c\|_\infty^2 e^{-\frac{\nu\beta}{2}m} $, $\Xi_2:= D_2\|c\|_\infty^2\frac{1}{e^{\frac{b}{\sigma h}}  -  1}$, $\Xi_3:= D_3\|c\|_\infty^2 e^{-\frac{2\pi^2\sigma^2 h^2 q^2}{b^2}}$ and $\Xi_4:= D_4 h\|\eta\|_\infty$.
%Let $D_1,D_2,D_3, D_4$ be the four terms as in \eqref{eq9}. 
By the assumptions \eqref{cond.N} - \eqref{cond.eta}, we have
 $$\Xi_k\le \frac{r\varepsilon}{16s}, \  k=1,2,3,4.$$ 
Hence,  \eqref{eq9} yields
\begin{equation}\label{Ej-uniform-small.eq}
\max_{1\le j\le J-1}\|E_j\|_{L^\infty([0,p_{j+1}-p_j])} 
\le \Xi_1 + \Xi_2 + \Xi_3 + \Xi_4 \le \frac{r\varepsilon}{4s} \le \frac{\varepsilon}{2}.
\end{equation}
Consequently,  for $1\le j\le J-1$, 
$$ E_j(\xi)\le \frac{r\varepsilon}{4s} \le \frac{\varepsilon}{2},  \ 0\le \xi\le p_{j+1}-p_j. $$

\smallskip
\noindent\textbf{Step 1: Verification of Condition $(\mathbf P)$.} 
For each anchor point $p_j$, we have $A_j=\mathcal A(p_j)$  by \eqref{eq:anchor-detector} and \eqref{eq:Aj-def}. 
Moreover, we have 
$$ E_j(0)=\bigl||f(p_j)|^2-A_j\bigr| \le \frac{\varepsilon}{2}\le \frac{\gamma^2}{2}, \ 1\le j\le J$$
by \eqref{eq:Ej-def}, \eqref{Ej-uniform-small.eq} and \eqref{local.error.def.eq1}. 
Since Step 2 of the algorithm ensures
$$ A_j\ge \widetilde\gamma=\frac{3}{2}\gamma^2, $$
it follows that
$$ |f(p_j)|^2\ge A_j - \bigl||f(p_j)|^2-A_j\bigr| \ge \gamma^2, \ 1\le j\le J.$$ 
Therefore,
$$ |f(p_j)|\ge \gamma, \ 1\le j\le J,$$
and thus $f$ satisfies Condition~$(\mathbf P)$ with anchor points
$$ -s=p_1<\cdots<p_J=s.$$

\smallskip
\noindent \textbf{Step 2: Prove the well-definedness of phase transition factor $\rho_j$ in \eqref{eq:rhoj-def}.}
  Let $1\le j\le J-1$ and $0\le \xi\le p_{j+1}-p_j$. By the definition of $R_j$, $A_j$, and $E_j$, we have
    \begin{align*}%\label{eq38}
        {}& \left| f(p_j + \xi)  - \frac{f(p_j)}{|f(p_j)|} R_j(\xi) \right|  = \left| \frac{\overline{f(p_j)}}{|f(p_j)|} f(p_j + \xi) - R_j(\xi) \right| \notag\\
       {}&  \le \left| \frac{\overline{f(p_j)}}{|f(p_j)|} f(p_j + \xi) - \frac{\overline{f(p_j)}}{\sqrt{A_j}}f(p_j+\xi)\right|+
\left|\frac{f_{-\xi}(p_j)}{\sqrt{A_j}} - R_j(\xi) \right| \notag\\
       {}&  \le \left| \frac{|f(p_j)|^2 - A_j}{\sqrt{A_j} (|f(p_j)| + \sqrt{A_j})} \right| \|f\|_{L^{\infty} (I)} + \frac{1}{\sqrt{A_j}} \left|\overline{f_\xi(p_j + \xi)} - \sqrt{A_j} R_j(\xi) \right| \notag\\
       {}&  =   \frac{E_j(0)}{\sqrt{A_j} (|f(p_j)| + \sqrt{A_j})}  \|f\|_{L^{\infty} (I)} + \frac{1}{\sqrt{A_j}} E_j(\xi).
    \end{align*}
    
As $|f(p_j)|\ge \gamma$ and $A_j\ge \frac{3}{2}\gamma^2>\gamma^2$, 
we have 
$$ \sqrt{A_j}\ge \gamma \ \ {\rm and} \ \ |f(p_j)|+\sqrt{A_j}\ge 2\gamma.$$  
Combine with \eqref{local.error.def.eq1} and \eqref{Ej-uniform-small.eq},  
we get
\begin{equation}\label{eq38}
\left|f(p_j+\xi)-\frac{f(p_j)}{|f(p_j)|}R_j(\xi)\right| \le \frac{\|f\|_{L^\infty(I)}}{2\gamma^2}E_j(0) +\frac{1}{\gamma}E_j(\xi) \le \frac{\varepsilon\|f\|_{L^\infty(I)}}{4\gamma^2} +\frac{\varepsilon}{2\gamma} \le\frac{3\gamma}{4}.
\end{equation}
Taking $\xi=p_{j+1}-p_j$ yields 
$$\left| f(p_{j+1})-\frac{f(p_j)}{|f(p_j)|}R_j(p_{j+1}-p_j)\right|\le \frac{3\gamma}{4}.$$ 
As $|f(p_{j+1})|\ge \gamma$, we have 
$$|R_j(p_{j+1}-p_j)|\ge  |f(p_{j+1})| - \left| f(p_{j+1}) - \frac{f(p_j)}{|f(p_j)|} R_j(p_{j+1} - p_j) \right|\ge \frac{\gamma}{4}.$$
Hence $R_j(p_{j+1}-p_j)\neq0$, and the phase factor
$$ \rho_j=\frac{R_j(p_{j+1}-p_j)}{|R_j(p_{j+1}-p_j)|} $$
is well defined for $1\le j\le J-1$.

 \smallskip
\noindent\textbf{Step 3: Global reconstruction error.}
Set
$$ \mu:=\left\lceil\frac{J}{2}\right\rceil \ \ {\rm and } \ \ \tau_\mu:=\frac{f(p_\mu)}{|f(p_\mu)|}\in\mathbb T. $$ 
We estimate $f-\tau_\mu\mathcal R$ separately on the right, middle, and left parts of the interval $I=[p_1,p_J]$.

\medskip
\noindent
\emph{\textbf{The right interval $(p_{\mu+1},p_J]$.}}
Let
$$ t=p_j+\xi\in(p_{\mu+1},p_J],$$
where $ j\in\{\mu+1,\dots,J-1\}$ and $0<\xi\le p_{j+1}-p_j$.  By the definition of the reconstruction $\mathcal R$, we have 
 \begin{align}\label{eq:right-interval-first}%\label{eq46}
      &  | f(t) {} - \frac{f(p_{\mu})}{|f(p_{\mu})|} \mathcal{R}(t) |  = \left| \frac{\overline{f(p_{\mu})}}{|f(p_{\mu})|} f(p_j+\xi) -  \rho_{\mu} \cdots \rho_{j-1} R_j(\xi) \right| \notag\\
        {}& \le \left| \frac{\overline{f(p_{\mu})} }{|f(p_{\mu})| } f(p_j+\xi) - \rho_{\mu} \cdots \rho_{j-1} \frac{1}{\sqrt{A_j}} \overline{f(p_j)} f(p_j+\xi) \right| + \left| \frac{1}{\sqrt{A_j}} \overline{f(p_j)} f(p_j+\xi) -  R_j(\xi) \right| \notag\\
        %{}& = \left| \frac{\overline{f(p_{\mu})}}{|f(p_{\mu})|} - \rho_{\mu} \cdots \rho_{j-1} \frac{1}{\sqrt{A_j}} \overline{f(p_j)}  \right| |f(p_j+\xi)| + \left| \frac{1}{\sqrt{A_j}} \overline{f_\xi (p_j + \xi)} - R_j(\xi) \right| \notag\\
        {}& \le \left| \frac{\overline{f(p_{\mu})}}{|f(p_{\mu})|} - \rho_{\mu} \cdots \rho_{j-1} \frac{\overline{f(p_j)}}{\sqrt{A_j}}  \right| \|f\|_{L^{\infty} (I)} + \frac{1}{\sqrt{A_j}} E_j(\xi) \notag\\
        {}& < \left| \frac{\overline{f(p_{\mu})}}{|f(p_{\mu})|} - \rho_{\mu} \cdots \rho_{j-1} \frac{\overline{f(p_j)}}{\sqrt{A_j}}   \right| \|f\|_{L^{\infty} (I)} +  \frac{\varepsilon}{2\gamma},
    \end{align}
    where the last inequality follows from \eqref{Ej-uniform-small.eq} and $A_j > \gamma^2$.

It remains to estimate 
$$ \left| \frac{\overline{f(p_\mu)}}{|f(p_\mu)|} - \rho_\mu\cdots\rho_{j-1}\frac{\overline{f(p_j)}}{\sqrt{A_j}}
\right|.$$
As $$ \left| \frac{\overline{f(p_\mu)}}{|f(p_\mu)|} - \rho_\mu\cdots\rho_{j-1}\frac{\overline{f(p_j)}}{\sqrt{A_j}}\right|
        \le \left|\frac{\overline{f(p_\mu)}}{|f(p_\mu)|}- \rho_{\mu} \frac{\overline{f(p_{\mu+1})}}{|f(p_{\mu+1})|} \right| + \left| \frac{\overline{f(p_{\mu+1})}}{|f(p_{\mu+1})|} - \rho_{\mu+1} \cdots \rho_{j-1} \frac{\overline{f(p_j)} }{\sqrt{A_j}} \right|,$$  
 we obtain
\begin{align}\label{eq:right-phase-split}
\left| \frac{\overline{f(p_\mu)}}{|f(p_\mu)|} - \rho_\mu\cdots\rho_{j-1}\frac{\overline{f(p_j)}}{\sqrt{A_j}} \right|
   \le {}& \sum_{k=\mu}^{j-1} \left| \frac{\overline{f(p_k)} f(p_{k+1}) }{|f(p_k) f(p_{k+1})|} - \rho_k \right| + \left| \frac{\overline{f(p_j)}}{|f(p_j)|} - \frac{\overline{f(p_j)}}{\sqrt{A_j}} \right| \notag\\
= {}& \sum_{k=\mu}^{j-1} \left| \frac{\overline{f(p_k)}f(p_{k+1})}{|f(p_k)f(p_{k+1})|} -\rho_k\right|
+\frac{E_j(0)}{\sqrt{A_j}\bigl(|f(p_j)|+\sqrt{A_j}\bigr)}. 
\end{align}

For each $k\in\{\mu,\dots,j-1\}$,  Lemma \ref{prop1} yields 
\begin{align}\label{eq:right-phase-step}%{eq41}
\left| \frac{\overline{f(p_k)}f(p_{k+1})}{|f(p_k)f(p_{k+1})|} -\rho_k\right|
&= \left|\frac{\overline{f(p_k)}f(p_{k+1})}{|f(p_k)f(p_{k+1})|}-
\frac{\overline{\mathcal G_{p_k}(p_{k+1}-p_k)}}{|\mathcal G_{p_k}(p_{k+1}-p_k)|}\right| \notag\\
&\le2\frac{\bigl|\overline{f(p_k)}f(p_{k+1})-\overline{\mathcal G_{p_k}(p_{k+1}-p_k)}\bigr|}{|f(p_k)f(p_{k+1})|}
=\frac{2E_k(p_{k+1}-p_k)}{|f(p_k)f(p_{k+1})|},
\end{align}
where the first equality follows by  \eqref{eq:Aj-def}, \eqref{eq:Rj-def}, \eqref{eq:rhoj-def} and \eqref{eq:Gp-def}.  
Substituting \eqref{eq:right-phase-step} into \eqref{eq:right-phase-split}, and combine  with $|f(p_k)|\ge \gamma$, $
A_j\ge \gamma^2$, \eqref{J-bound.eq} and  \eqref{Ej-uniform-small.eq}, 
we obtain
\begin{align}\label{eq:right-phase-final}
\left| \frac{\overline{f(p_\mu)}}{|f(p_\mu)|} - \rho_\mu\cdots\rho_{j-1}\frac{\overline{f(p_j)}}{\sqrt{A_j}}\right|
&\le \sum_{k=\mu}^{j-1}\frac{2E_k(p_{k+1}-p_k)}{\gamma^2} + \frac{E_j(0)}{2\gamma^2} \notag\\
& \le \frac{J-1}{2} \frac{2}{\gamma^2} \frac{\varepsilon}{J-2}  + \frac{\varepsilon}{4 \gamma^2} \le  \frac{2\varepsilon}{\gamma^2} + \frac{\varepsilon}{4 \gamma^2}  =  \frac{9\varepsilon}{4 \gamma^2} . 
\end{align}
Combining \eqref{eq:right-interval-first} and \eqref{eq:right-phase-final}, we obtain
\begin{equation}\label{eq:right-interval-final}
\|f-\tau_\mu\mathcal R\|_{L^\infty((p_{\mu+1},p_J])} \le
\frac{9\varepsilon\|f\|_{L^\infty(I)}}{4\gamma^2} + \frac{\varepsilon}{2\gamma}.
\end{equation}

\medskip
\noindent
\emph{\textbf{The middle interval $[p_\mu,p_{\mu+1}]$.}}  
Let $ t=p_\mu+\xi\in[p_\mu,p_{\mu+1}], \ 
0\le \xi\le p_{\mu+1}-p_\mu$. 
No phase propagation is needed. Therefore, by the same estimate as in Step 2,
$$ |f(t)-\tau_\mu\mathcal R(t)| \le \frac{E_\mu(0)}{2\gamma^2}\|f\|_{L^\infty(I)}
+ \frac{E_\mu(\xi)}{\gamma}
\le \frac{\varepsilon}{4\gamma^2}\|f\|_{L^\infty(I)} + \frac{\varepsilon}{2\gamma}.$$ 
% Then 
% \begin{align*}
% |f(t)-\tau_\mu\mathcal R(t)| &=
% \left|\frac{\overline{f(p_\mu)}}{|f(p_\mu)|}f(p_\mu+\xi)-R_\mu(\xi)\right| \\
% &\le\frac{E_\mu(0)}{\sqrt{A_\mu}\bigl(|f(p_\mu)|+\sqrt{A_\mu}\bigr)}\|f\|_{L^\infty(I)}
% +\frac{1}{\sqrt{A_\mu}}E_\mu(\xi) \\
% &\le  \frac{\varepsilon}{2 \gamma^2}   \|f\|_{L^{\infty} (I)} + \frac{\varepsilon}{\gamma},
% \end{align*}
Therefore 
\begin{equation}\label{eq:middle-interval-final}
\|f-\tau_\mu\mathcal R\|_{L^\infty([p_\mu,p_{\mu+1}])}
\le\frac{ \varepsilon}{4 \gamma^2} \|f\|_{L^{\infty} (I)} + \frac{ \varepsilon}{2\gamma}.
\end{equation}

\medskip
\noindent
\emph{\textbf{The left interval $[p_1, p_{\mu})$.}} 
The same argument as for the right interval, with the phase propagating  from $p_\mu$ to the left,  gives
\begin{equation}\label{eq:left-interval-final}
\|f-\tau_\mu\mathcal R\|_{L^\infty([p_1,p_\mu))}
\le \frac{9\varepsilon\|f\|_{L^\infty(I)}}{4\gamma^2} + \frac{\varepsilon}{2\gamma}.
\end{equation}

Finally, combining \eqref{eq:right-interval-final}, \eqref{eq:middle-interval-final}, and \eqref{eq:left-interval-final}, we arrive at    
\begin{equation*}
    \min_{\tau\in\mathbb T}\|f-\tau\mathcal R\|_{L^\infty(I)} \le
  \|f-\tau_\mu\mathcal R\|_{L^\infty(I)} \le \frac{9\varepsilon}{4\gamma^2}\|f\|_{L^\infty(I)}+\frac{\varepsilon}{2\gamma} \le \frac{11\varepsilon}{4}\frac{\max\{1,\|f\|_{L^\infty(I)}\}}{\min\{\gamma,\gamma^2\}} = \frac{11}{4} \varepsilon \kappa_{f,\gamma},  
\end{equation*}
where $$ \kappa_{f,\gamma}:=\frac{\max\{1,\|f\|_{L^\infty(I)}\}}{\min\{\gamma,\gamma^2\}}.$$

If $J=2$, then  $p_1=-s$, $p_2=s$, $\mu=1$, and the phase propagation step is not needed. In this case,
$$\mathcal R(t)=R_1(t-p_1), \  t\in I=[p_1,p_2].$$ 
By the same argument used in Steps 1 and 2, we  know that 
$$|f(p_1)|\ge \gamma, \ \  A_1\ge \gamma^2,$$
 and $\rho_1$  is well defined.  
  Set
$$ \tau_1:=\frac{f(p_1)}{|f(p_1)|}\in\mathbb T.$$
For $t=p_1+\xi\in I, 0\le \xi\le p_2-p_1$, the middle-interval estimate gives
$$ |f(t)-\tau_1\mathcal R(t)| \le \frac{\varepsilon}{4\gamma^2}\|f\|_{L^\infty(I)}
+\frac{\varepsilon}{2\gamma}. $$ 
Hence 
$$ \min_{\tau\in\mathbb T}\|f-\tau\mathcal R\|_{L^\infty(I)} \le \frac{11}{4}\varepsilon\kappa_{f,\gamma}$$
also holds for $J=2$. The proof is complete.

\end{proof}

\begin{remark}[Noiseless reconstruction]
    Under the same setting as Theorem \ref{alg-stability.thm} with $\eta = 0$, the noise condition is vacuously satisfied. Hence, if the sampling lattice satisfies conditions \eqref{cond.N} - \eqref{cond.H}, then 
    \begin{align*}
        \min_{\tau\in\mathbb{T}}\|f - \tau\mathcal{R}\|_{L^\infty(I)} \le \frac{33}{16}\,\varepsilon\,\kappa_{f,\gamma}.
    \end{align*}
    In particular, by choosing $N$, $1/h$, and $H$ sufficiently large, the right-hand side can be made arbitrarily small, showing that Algorithm \ref{alg:stlct} achieves arbitrarily accurate reconstruction from exact phaseless STLCT samples on a sufficiently fine lattice.
\end{remark}

\begin{remark}[Comparison with \cite{grohs2024stable}]
    For the special case $\mathbf{A} = \begin{pmatrix} 0 & 1 \\ -1 & 0 \end{pmatrix}$, a similar result has been shown in \cite[Theorem 4.6]{grohs2024stable}. Note that our stability constant in \eqref{alg.robust.thm.eq2} depends linearly on $\varepsilon$, $\max\{1,\|f\|_{L^{\infty}(I)}\}$ and the reciprocal of $\min \{\gamma, \gamma^2\}$, whereas the stability constant in \cite[Eq.(48)]{grohs2024stable} depends linearly on $\varepsilon+\varepsilon^2$, $\max\{1, \|f\|_{L^{\infty}(I)}^2\}$ and the reciprocal of $\min \{\gamma, \gamma^5\}$. Especially when $\varepsilon$ and $\|f\|_{L^{\infty} (I)}$ are large, or $\gamma$ is small, our result improves the one in \cite{grohs2024stable}.
\end{remark}

\subsection{Numerical simulations}

We show the effectiveness of  Theorem \ref{alg-stability.thm} by applying Algorithm \ref{alg:stlct} to recover a complex-valued Gaussian shift-invariant signal from noisy phaseless STLCT samples. The experiment is set in a non-degenerate STLCT regime with $a \ne 0$, which goes beyond the Gabor case studied in \cite{grohs2024stable}.

\medskip
\noindent\textbf{Setup}: Let $\varphi(t) = e^{-t^2/(2\sigma^2)}$ be the Gaussian generator with $\sigma = 1/\sqrt{2\pi}$, and let
\[
f = \sum_{n=-n_0}^{n_0} c_n\,\varphi(\cdot - \beta n) \in V_\beta^\infty(\varphi)
\]
with step-size $\beta=1$, index bound $n_0=45$, and complex coefficients $c_n = c_n^{\rm re} + i\,c_n^{\rm im}$ with $c_n^{\rm re}, c_n^{\rm im}$ drawn i.i.d.\ uniformly from $[-6,6]$. The signal is considered on $I = [-s, s]$ with $s=40$. The LCT matrix $\mathbf{A} = \begin{pmatrix} 2 & 3 \\ 1 & 2 \end{pmatrix}$ ($a=2$, $b=3$, $c=1$, $d=2$, $ad-bc=1$) is chosen as a  non-degenerate instance. 
The noisy phaseless STLCT samples are
\begin{align}\label{gaussnoise.data.def}
		Y_{n,k} = |\mathcal{S}_{\check{\varphi}}^{\mathbf{A}} f (\frac{\beta}{2} n , hk)|^2 + \eta_{n, k}, \quad (n,k) \in \{ -N, \dots, N \} \times \{ -H, \dots, H \},
\end{align}
where $\eta_{n,k}\sim \mathcal{N}(0,\delta^2)$ is i.i.d.\ Gaussian noise.
Set the noise level $\delta = 0.001$ and the  algorithm inputs  $r = \frac32$ and $\tilde{\gamma} = \frac12$.

\medskip
\noindent\textbf{Sampling parameters}:  
The finite sampling lattice is chosen as 
$$X_{N,H,h} = \frac{\beta}{2}\{-N,\dots,N\} \times h\{-H,\dots,H\}$$
with the parameters calibrated according to the sufficient conditions
\eqref{cond.N}--\eqref{cond.H}.

\begin{itemize}

\item \emph{Time-domain truncation}: The leading term in \eqref{cond.N} is 
  $\lceil\frac{2}{\beta}(s+\frac{r}{2})\rceil = 82$. We take $N = 90$.

\item \emph{Frequency step}: We set $h = \frac{1}{16}$,
in accordance with  the lower bound on $1/h$ in \eqref{cond.h}.

\item \emph{Frequency truncation}: For this choice of $N$ and $h$, the leading term in \eqref{cond.H} is 
  $\lceil\frac{|a|\beta N}{2h}\rceil = 1440$. We take $H = 2000$.
\end{itemize}

Thus the reconstruction uses $181 \times 4001$ phaseless STLCT samples. 
The relatively large value of $H=2000$ reflects the dependence of the sampling condition \eqref{cond.H} on the STLCT parameter matrix 
$$\mathbf A=
\begin{pmatrix}
2 & 3\\
1 & 2
\end{pmatrix}.$$
 Indeed, the first term $\left\lceil \frac{|a|\beta N}{2h}\right\rceil$ in \eqref{cond.H} is driven by the parameter $a=2$, while the second term depends on $b=3$. By contrast, in the Gabor case
$$\mathbf A=
\begin{pmatrix}
0 & 1\\
-1 & 0
\end{pmatrix},$$
 with $a=0$ and $b=1$. Hence the first term in \eqref{cond.H} vanishes and the second term is significantly smaller. In that case, $H=75$ already suffices in the corresponding numerical experiment of \cite{grohs2024stable}.

\medskip
\noindent\textbf{Results}: Figure \ref{fig.pj} illustrates the selection of anchor points in Step 2 of Algorithm \ref{alg:stlct}. The approximation $\mathcal A$ defined in \eqref{mathcalA.alg.eq1}, computed from $181\times 4001$ noisy phaseless STLCT samples, is visually almost indistinguishable from the true squared magnitude $|f|^2$. This indicates that the noise level $\delta=0.001$ has only a minor effect on the magnitude estimation. The algorithm selects $57$ anchor points satisfying
$$\mathcal A(p_j)\ge \widetilde{\gamma}=\frac12$$
and
$$p_{j+1}-p_j\le r=\frac32, \ \  p_{j+2}-p_j\ge r. $$
As shown in the figure, these anchor points are well distributed over the whole interval and avoid the regions where $|f|^2$ is close to zero, in agreement with Condition $(\mathbf P)$.

Figure \ref{fig.recon} compares the original signal $f$ with the reconstruction $\mathcal R$ in \eqref{rec.alg.eq} after optimal global phase alignment. The reconstructed signal $\mathcal R$ agrees closely with $f$ in both its real and imaginary parts over most of the interval, which is consistent with the robustness estimate in Theorem \ref{alg-stability.thm}. These numerical simulations demonstrate the effectiveness and stability of the proposed algorithm for reconstructing Gaussian shift-invariant signals from finite noisy phaseless STLCT samples.

\begin{figure}[htp]
	\centering
	\includegraphics[width=0.8\textwidth]{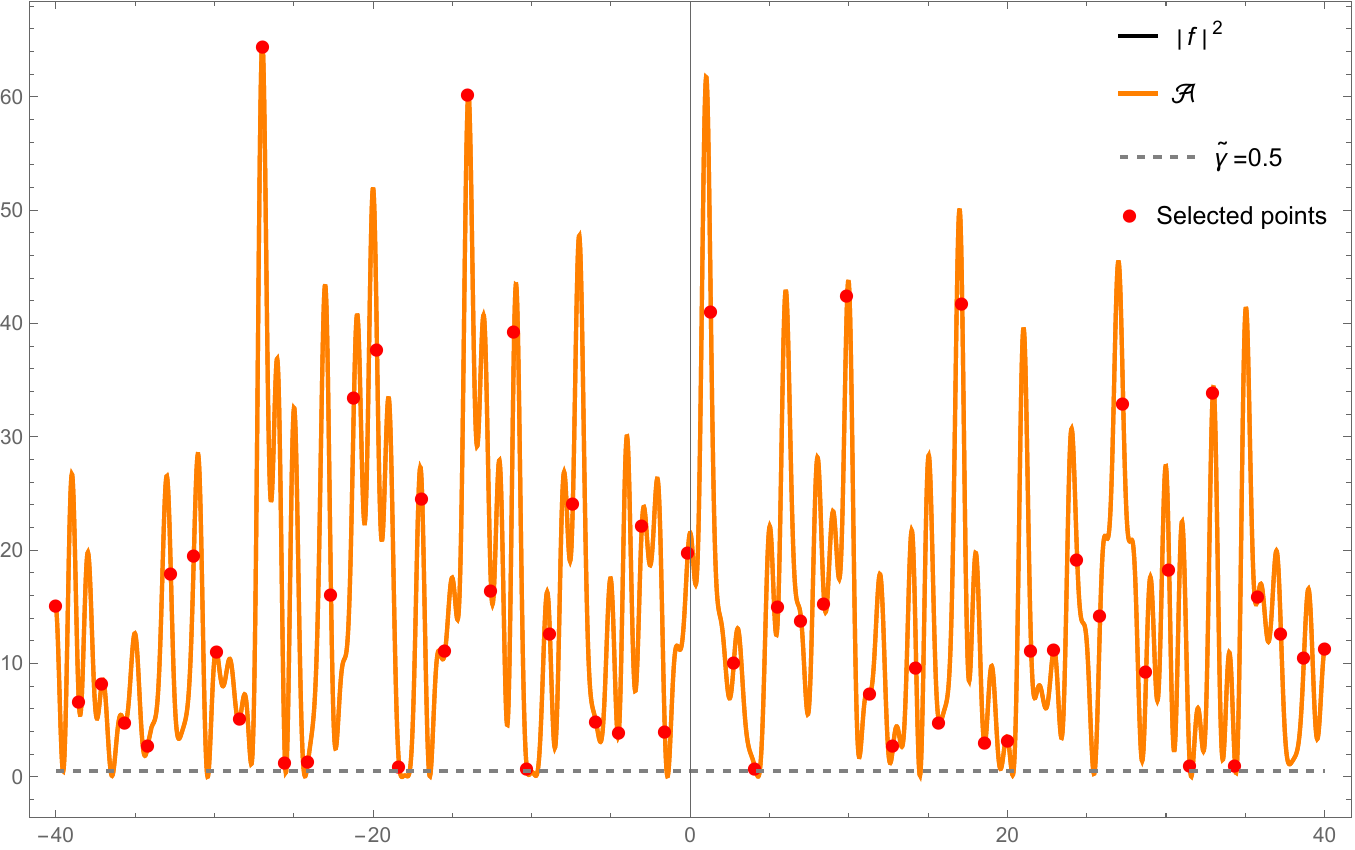}
	\caption{Anchor points selection for Algorithm \ref{alg:stlct}. The black curve shows $|f|^2$ and the orange curve shows the approximation $\mathcal A$ computed from noisy samples. The two curves nearly coincide. Red dots indicate the selected anchor points $(p_j, \mathcal A(p_j))$ satisfying $p_{j+1} - p_j \le r = \frac32$, $p_{j+2} - p_j \ge r$ and $\mathcal A(p_j) \ge \tilde{\gamma} = \frac12$.
   }
	\label{fig.pj}
	%\textit{Note: This is a sine curve.}
\end{figure}

\begin{figure}[htp]
   \centering
   \begin{minipage}{0.48\textwidth}
       \centering
       \includegraphics[width=\textwidth]{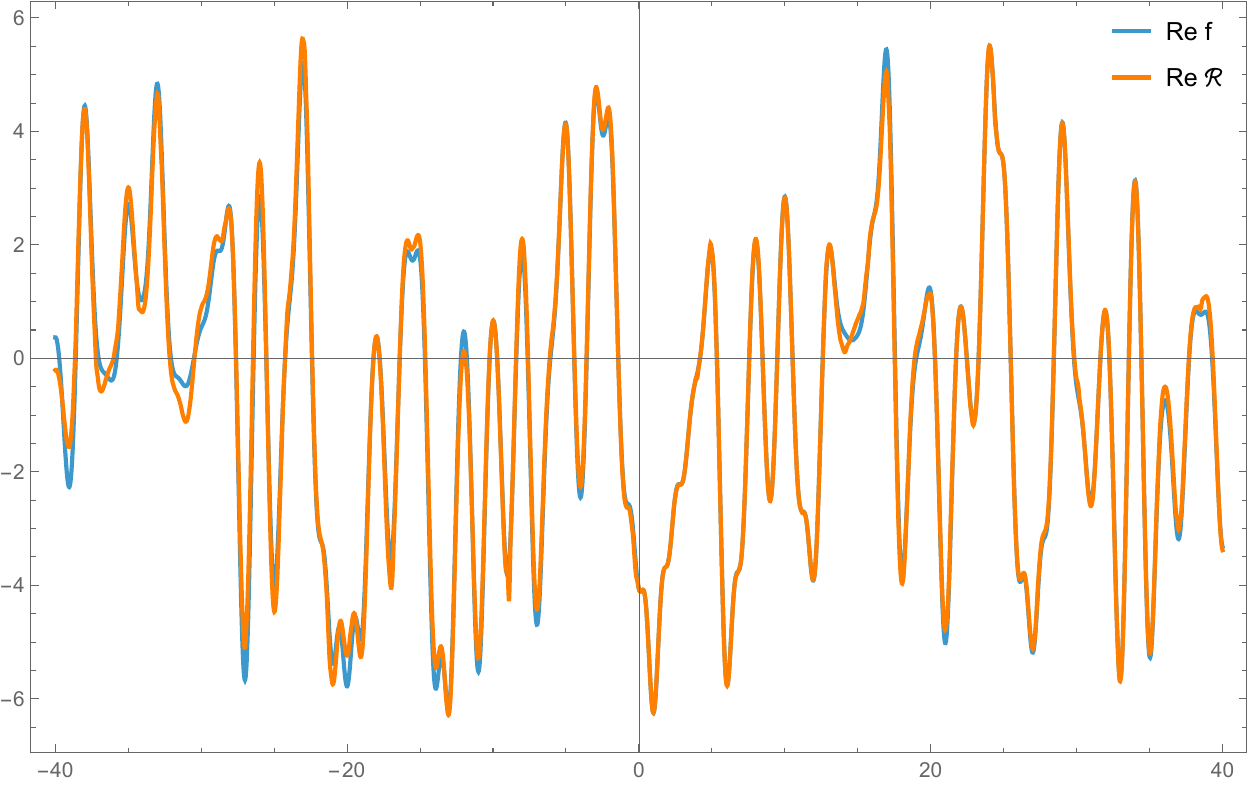}
       
   \end{minipage}
   \hfill
   \begin{minipage}{0.48\textwidth}
       \centering
       \includegraphics[width=\textwidth]{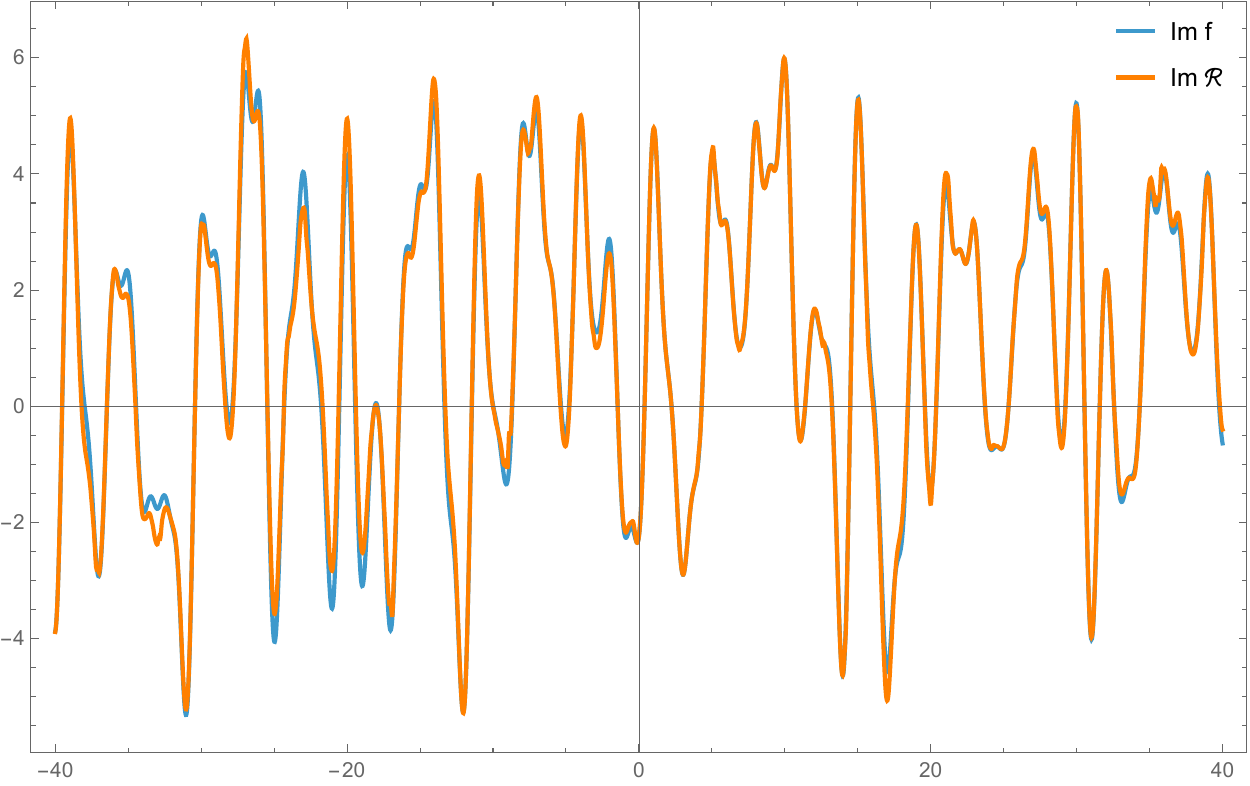}
       
   \end{minipage}
   \caption{Comparison of the original signal $f$ and the reconstructed signal $\mathcal{R}$ (up to a global phase). Left: real parts. Right: imaginary parts. The reconstruction is obtained from $181 \times 4001$ noisy phaseless STLCT samples with noise level $\delta = 0.001$.}
   \label{fig.recon}
\end{figure}

\begin{appendices}

\section{Proofs of some results}\label{app1}

\subsection{Proof of Proposition \ref{prop3}}\label{proof.prop3}

We  need the following lemma, which is a generalization of \cite[Proposition 9.4.2]{christensen2016introduction}.

\begin{lemma}\cite[Theorem 2.3]{grohs2024stable}\label{lem3}
	Let $\beta > 0$, $(h(\cdot-\beta n))_n$ be a Riesz basis for the shift-invariant space $V_{\beta}^2 (h)$ and $D= \{t\in \R: \Psi_{\beta}(t)\ne 0\}$, where $\Psi_{\beta}$ is the $\frac{1}{\beta}$-periodic map defined as in \eqref{psi.beta.eq}. Then $\tilde{h}= \theta$ is the dual generator of $h$, where
	\begin{align*}
		\hat{\theta}(\gamma) := 
		\begin{cases} 
			\frac{\beta \hat{h}(\gamma)}{\Psi_{\beta}(\gamma)} & \text{if } \gamma \in D, \\
			0 & \text{if } \gamma \notin D.
		\end{cases}
	\end{align*}
\end{lemma}

We are ready to prove Proposition \ref{prop3}.

\begin{proof}[Proof of Proposition \ref{prop3}]
        The $\frac{1}{\beta}$-periodization of $\varphi$ with step-size $\beta$ is given by
	\begin{align*}
		\Psi_{\beta} (t) {}& = \sum_{n \in \Z} | \mathcal{F}\varphi (t + \frac{n}{\beta})|^2 = 2 \pi \sigma^2 \sum_{n \in \Z} e^{- \frac{4 \pi^2 \sigma^2}{\beta^2} (\beta t + n)^2 } \\
		{}& = \sqrt{ \pi} \sigma \beta \sum_{n \in \Z} e^{-\frac{\beta^2 n^2}{4 \sigma^2}} e^{2 \pi i n \beta t} = \sqrt{ \pi} \sigma \beta \vartheta_3(\pi \beta t , e^{-\frac{\beta^2}{4 \sigma^2}}),
	\end{align*}
        where the third equality follows from Poisson's summation formula. Using the properties of the Jacobi theta function of the third kind, we have $\Psi_{\beta} (t)$ is minimal at point $t = \frac{1}{2 \beta}$ and maximal at point $t = 0$. Thus, $(\varphi(\cdot - \beta n))_n$ forms a Riesz basis for $V_{\beta}^{2}(\varphi)$.
        %with optimal lower bound $\sqrt{2 \pi} \sigma \vartheta_3(\frac{\pi}{2} , e^{-\frac{\beta^2}{2 \sigma^2}})$ and optimal upper bound $\sqrt{2 \pi} \sigma \vartheta_3(0 , e^{-\frac{\beta^2}{2 \sigma^2}})$.
        
	The $\frac{2}{\beta}$-periodization of $\varphi_\xi$ with step-size $\frac{\beta}{2}$ is given by
	\begin{align*}
		\Psi_{\beta/2}(t) {}& = \sum_{n\in \Z} |\mathcal{F} \varphi_\xi(t + \frac{2n}{\beta})|^2 = \pi \sigma^2 e^{-\frac{\xi^2}{2\sigma^2}} \sum_{n\in \Z} e^{-\frac{8\pi^2 \sigma^2}{\beta^2} (\frac{\beta t}{2} + n)^2} \\
		{}& = \frac{\sigma \beta \sqrt{\pi}}{2 \sqrt{2}} e^{-\frac{\xi^2}{2\sigma^2}} \sum_{n\in \Z} e^{-\frac{\beta^2 n^2}{8\sigma^2}} e^{\pi in \beta t} = \frac{\sigma \beta \sqrt{\pi}}{2\sqrt{2}} e^{-\frac{\xi^2}{2\sigma^2}} \vartheta_3(\frac{\pi \beta t}{2},e^{-\frac{\beta^2}{8\sigma^2}}),
	\end{align*}
        where the third equality follows from Poisson's summation formula. Using the properties of the Jacobi theta function of the third kind, we have $\Psi_{\beta/2}(t)$ is minimal at point $t = \frac{1}{\beta}$ and maximal at point $t = 0$. Thus, $(\varphi_\xi(\cdot - \frac{\beta}{2} n))_n$ forms a Riesz basis for $V_{\beta/2}^{2}(\varphi_\xi)$.
        %with optimal lower bound $\sigma \sqrt{\pi} e^{-\frac{\xi^2}{2\sigma^2}} \vartheta_3(\frac{\pi}{2},e^{-\frac{\beta^2}{4\sigma^2}})$ and optimal upper bound $\sigma \sqrt{\pi} e^{-\frac{\xi^2}{2\sigma^2}} \vartheta_3(0 , e^{-\frac{\beta^2}{4\sigma^2}})$.

        By $\mathcal{F} \varphi_\xi (t) = \sqrt{\pi} \sigma e^{-\frac{\xi^2}{4 \sigma^2}} e^{- \pi^2 \sigma^2 t^2} e^{-\pi i t\xi}$ and Lemma \ref{lem3}, we obtain that the dual generator of $\varphi_\xi$ is
	\begin{align*}
		\widetilde{\varphi_\xi} (t) = \sqrt{2}e^{\frac{\xi^2}{4\sigma^2}} \mathcal{F}^{-1} \frac{e^{-\pi^2 \sigma^2 t^2} e^{-\pi i t\xi}}{\vartheta_3(\frac{\beta}{2}\pi t, e^{-\frac{\beta^2}{8\sigma^2}})}
        = \sqrt{2} e^{\frac{\xi^2}{4\sigma^2}} T_{\frac{\xi}{2}} \mathcal{F}^{-1} \frac{e^{-\pi^2 \sigma^2 t^2}}{\vartheta_3(\frac{\beta}{2}\pi t, e^{-\frac{\beta^2}{8\sigma^2}})} = \sqrt{2} e^{\frac{\xi^2}{4\sigma^2}} T_{\frac{\xi}{2}} \mathcal{F}^{-1} \Lambda(t).
	\end{align*}
    \end{proof}

\subsection{Proof of Lemma \ref{lem4}}\label{proof.lem4} 
	Fix $x\in \R$. Write  $f^{\omega,x}(t)= f(t)\overline{\omega(t-x)} e^{i\pi \frac{a}{b} t^2}$. As 
    \begin{align*}
        \int_{\R} |f^{\check{\omega},x} (t)|^2 dt {}& = \int_{\R} | f(t) e^{-i\pi \frac{a}{b} (t-x)^2} \overline{\omega(t-x)} e^{i\pi \frac{a}{b} t^2} |^2 dt \\
        {}& = \int_{\R} | f(t) \overline{\omega(t-x)} |^2 dt \le \|f\|_{\infty}^2 \int_{\R} |\omega(t-x)|^2 dt < \infty,
    \end{align*}
%   then $f^{\check{\omega},x}, \mathcal{F}f^{\check{\omega},x}\in L^2(\R)$, and then $S_{\check{\omega}}^{\textbf{A}} f(x,\cdot)\in L^2(\R)$ by \eqref{eq59}. %$\| \mathcal{F}f^{\check{g},x} \|_{L^2 (\R)} = \| f^{\check{g},x} \|_{L^2 (\R)} < \infty$. 
%   And
%    \begin{align*}
%        \int_{\R} | S_{\check{g}}^{\textbf{A}} f(x,\xi) |^2 dw {}& = \int_{\R} | \frac{1}{\sqrt{i b}} e^{i\pi \frac{d}{b} \xi^2} \mathcal{F} f^{\check{g},x} (\frac{\xi}{b}) |^2 dw \\
%        {}& = \int_{\R} | \mathcal{F} f^{\check{g},x} (\frac{\xi}{b}) |^2 d\frac{\xi}{b} = \| \mathcal{F}f^{\check{g},x} \|_{L^2 (\R)}^2 < \infty,
%    \end{align*}
%    where the first inequality follows from \eqref{eq59}.
%    Then $|S_{\check{\omega}}^{\textbf{A}} f(x,\cdot)|^2\in L^1(\R)$ and the Fourier transform of map $t\mapsto |S_{\check{\omega}}^{\textbf{A}} f(x,t)|^2$ is continuous and pointwise defined. 
   then $f^{\check{\omega},x}\in L^2(\R)$, and hence $\mathcal{F}f^{\check{\omega},x}\in L^2(\R)$ by Plancherel's theorem. It follows from \eqref{eq59} that $ S_{\check{\omega}}^{\mathbf A}f(x,\cdot)\in L^2(\R)$.  Consequently, $ |S_{\check{\omega}}^{\mathbf A}f(x,\cdot)|^2\in L^1(\R)$, 
and therefore the Fourier transform of the function
$$ t\mapsto |S_{\check{\omega}}^{\mathbf A}f(x,t)|^2 $$
is well defined, continuous, and pointwise defined on $\R$.
    
    By \eqref{eq59} and $\mathcal{F}^{-1} \overline{f} (\xi) = \overline{\mathcal{F} f(\xi)}$, we have
	\begin{align}\label{sgf.def}
		|S_{\check{\omega}}^{\textbf{A}} f(x,b\xi)|^2 {}& = \frac{1}{b} (\mathcal{F}f^{\check{\omega},x})(\xi) \overline{(\mathcal{F}f^{\check{\omega},x})(\xi)}= \frac{1}{b} (\mathcal{F}f^{\check{\omega},x})(\xi) (\mathcal{F}^{-1} \overline{f^{\check{\omega},x}}) (\xi)\nonumber\\  
{}&= \frac{1}{b} \mathcal{F}^{-1} ( \mathcal{F}^2 f^{\check{\omega},x}\ast \overline{f^{\check{\omega},x}} ) (\xi),
	\end{align}
    where the last equality follows from the convolution theorem.
	The convolution $\mathcal{F}^2 f^{\check{\omega},x}\ast \overline{f^{\check{\omega},x}}$ can be written as
	\begin{align}\label{conv.eq.1}
		\mathcal{F}^2 f^{\check{\omega},x}\ast \overline{f^{\check{\omega},x}} (\xi) {}& = \int_{\R} \mathcal{F}^2 f^{\check{\omega},x}(t) \overline{f^{\check{\omega},x}(\xi-t)} dt
		= \int_{\R} f^{\check{\omega},x}(-t) \overline{f^{\check{\omega},x}(\xi-t)} dt\nonumber \\ 
		{}& \overset{t=\xi-t}{=} \int_{\R} f^{\check{\omega},x}(t-\xi) \overline{f^{\check{\omega},x}(t)} dt 
		 = e^{\frac{-2\pi iax\xi}{b}} \int_{\R} f_{\xi}(t) \overline{T_x \omega_{\xi} (t)} dt\nonumber \\ 
		{}& = e^{\frac{-2\pi iax\xi}{b}} \langle f_{\xi}, T_x \omega_{\xi} \rangle.
	\end{align}
	Combine \eqref{sgf.def} and \eqref{conv.eq.1}, we have 
	 \begin{equation*}\label{eq1}
		e^{\frac{-2\pi iax\xi}{b}} \langle f_\xi, T_x \omega_\xi \rangle=(\mathcal{F} |S_{\check{\omega}}^{\textbf{A}} f(x,\cdot)|^2)(\frac{\xi}{b}), 
	\end{equation*}
	which completes the proof.

\subsection{Proof of Lemma \ref{prop1}}\label{proof.prop1}
Since $|z_1|\neq 0$ and $|z_2|\neq 0$,  we have   \begin{align*}
        |z_1 - z_2| = \left| |z_1| \frac{z_1}{|z_1|} - |z_2| \frac{z_2}{|z_2|} \right|  {}& \ge \left| |z_1| \frac{z_1}{|z_1|} - |z_1| \frac{z_2}{|z_2|} \right| - \left| |z_1| \frac{z_2}{|z_2|} - |z_2| \frac{z_2}{|z_2|} \right| \\
        {}& = |z_1| \left| \frac{z_1}{|z_1|} - \frac{z_2}{|z_2|} \right| - \big| |z_1| - |z_2| \big|.
    \end{align*}
Therefore, 
    \begin{align*}
        \left| \frac{z_1}{|z_1|} - \frac{z_2}{|z_2|} \right|  \le \frac{1}{|z_1|} (|z_1 - z_2| + \big| |z_1| - |z_2| \big|) 
         \le 2\frac{|z_1 - z_2|}{|z_1|},
    \end{align*}
    where the last inequality follows by the triangle inequality.

\subsection{Proof of Lemma~\ref{lem:entire-stlct-magnitude}}\label{proof.lem.entire-stlct-magnitude} 
For every $(x,t)\in\R^2$, we have
\begin{align*}
S_{\check{\varphi}}^{\mathbf A}f(x,t) &= \frac{1}{\sqrt{ib}}e^{i\pi \frac{d}{b}t^2}
\int_{\R} f(\xi)\overline{\check{\varphi}(\xi-x)}e^{i\pi\frac{a}{b}\xi^2}e^{-2\pi i\frac{t}{b}\xi}\,d\xi \\
&=\frac{1}{\sqrt{ib}}e^{i\pi \frac{d}{b}t^2}\int_{\R}
\sum_{k\in\Z}c_k \varphi(\xi-\beta k)e^{-i\pi\frac{a}{b}(\xi-x)^2}\varphi(\xi-x)e^{i\pi\frac{a}{b}\xi^2}
e^{-2\pi i\frac{t}{b}\xi}\,d\xi \\
&=\frac{1}{\sqrt{ib}}e^{i\pi \frac{d}{b}t^2}e^{-i\pi\frac{a}{b}x^2}\sum_{k\in\Z}c_k e^{-\frac{(x-\beta k)^2}{4\sigma^2}}\int_{\R}e^{-\frac{(\xi-\frac{x+\beta k}{2})^2}{\sigma^2}}
e^{-2\pi i\frac{t-ax}{b}\xi}\,d\xi.
\end{align*}
Evaluating the Gaussian Fourier integral yields
\begin{align*}
S_{\check{\varphi}}^{\mathbf A}f(x,t) =
\frac{\sigma\sqrt{\pi}}{\sqrt{ib}}e^{i\pi\frac{d}{b}t^2}e^{-i\pi\frac{a}{b}x^2}
e^{-\frac{\pi^2\sigma^2}{b^2}(t-ax)^2}\sum_{k\in\Z}c_k e^{-\frac{(x-\beta k)^2}{4\sigma^2}}
e^{-\pi i\frac{(x+\beta k)(t-ax)}{b}}.
\end{align*}
Hence
\begin{align*}
|S_{\check{\varphi}}^{\mathbf A}f(x,t)|^2&= \frac{\pi\sigma^2}{b}
e^{-\frac{2\pi^2\sigma^2}{b^2}(t-ax)^2}\sum_{n\in\Z}\sum_{k\in\Z}c_n\overline{c_k}
e^{-\frac{(x-\beta n)^2}{4\sigma^2}}e^{-\frac{(x-\beta k)^2}{4\sigma^2}}e^{\pi i\frac{\beta (k-n)(t-ax)}{b}}.
\end{align*}
Setting $\ell=k-n$, we obtain
\begin{align*}
|S_{\check{\varphi}}^{\mathbf A}f(x,t)|^2&=\frac{\pi\sigma^2}{b}e^{-\frac{2\pi^2\sigma^2}{b^2}(t-ax)^2}\sum_{\ell\in\Z}\left(\sum_{n\in\Z}c_n\overline{c_{n+\ell}}e^{-\frac{(x-\beta n)^2}{4\sigma^2}}
e^{-\frac{(x-\beta (n+\ell))^2}{4\sigma^2}}\right)e^{-\pi i\frac{\beta ax}{b}\ell}e^{\pi i\frac{\beta t}{b}\ell}.\end{align*}
% Define
% \begin{equation}\label{eq:breve-r-def-proof}
% \breve r_\ell^x:=e^{-\pi i\frac{\beta ax}{b}\ell}\sum_{n\in\Z}c_n\overline{c_{n+\ell}}e^{-\frac{(x-\beta n)^2}{4\sigma^2}}e^{-\frac{(x-\beta (n+\ell))^2}{4\sigma^2}}, \ \ell\in\Z.
% \end{equation}
Then
$$ |S_{\check{\varphi}}^{\mathbf A}f(x,t)|^2=\frac{\pi\sigma^2}{b}e^{-\frac{2\pi^2\sigma^2}{b^2}(t-ax)^2}V_x(t),$$
where 
\begin{equation}\label{V_x.def.appx}
V_x(t):=\sum_{\ell\in\Z}\breve r_\ell^x e^{\pi i\frac{\beta\ell}{b}t}\end{equation}
with  coefficients $\breve r_\ell^x$ defined in \eqref{trig.coef.lem.def.eq1}.

It remains to estimate the coefficients $\breve r_\ell^x$. By the definition of  $\breve r_\ell^x$ in \eqref{trig.coef.lem.def.eq1}, we have 
\begin{align*}
|\breve r_\ell^x| &\le\|c\|_\infty^2 \sum_{n\in\Z}e^{-\frac{(x-\beta n)^2}{4\sigma^2}}e^{-\frac{(x-\beta(n+\ell))^2}{4\sigma^2}}.
\end{align*}
Using
$ (x-\beta n)^2+(x-\beta(n+\ell))^2 =
2\Bigl(x-\beta n-\frac{\beta\ell}{2}\Bigr)^2+\frac{\beta^2\ell^2}{2}$, 
we obtain
\begin{align*}
|\breve r_\ell^x| &\le \|c\|_\infty^2 e^{-\frac{\beta^2\ell^2}{8\sigma^2}}
\sum_{n\in\Z} e^{-\frac{\beta^2}{2\sigma^2}\left(n+\frac{\ell}{2}-\frac{x}{\beta}\right)^2}.
\end{align*}
Since the Gaussian is bounded by $1$ and integrable, we have
$$ \sup_{\alpha\in\R}\sum_{n\in\Z}e^{-\frac{\beta^2}{2\sigma^2}(n+\alpha)^2}
\le 1+\int_{\R}e^{-\frac{\beta^2}{2\sigma^2}u^2} du = 1+\frac{\sigma\sqrt{2\pi}}{\beta}.$$ 
Therefore,
$$ |\breve r_\ell^x| \le \|c\|_\infty^2 \left(1+\frac{\sigma\sqrt{2\pi}}{\beta}\right) e^{-\frac{\beta^2\ell^2}{8\sigma^2}}, \ell\in\Z.$$

Finally, let $K\subset\C$ be compact and set $
M_K:=\sup_{z\in K}|\Im z|<\infty$. 
Then for every $z\in K$,
\begin{align*}
\sum_{\ell\in\Z}\left|\breve r_\ell^x e^{\pi i\frac{\beta\ell}{b}z}\right|
&\le \|c\|_\infty^2 \left(1+\frac{\sigma\sqrt{2\pi}}{\beta}\right) \sum_{\ell\in\Z}
e^{-\frac{\beta^2\ell^2}{8\sigma^2}} e^{\pi \frac{\beta}{b}|\ell|\,|\Im z|} \\
&\le \|c\|_\infty^2 \left(1+\frac{\sigma\sqrt{2\pi}}{\beta}\right) \sum_{\ell\in\Z}
e^{-\frac{\beta^2\ell^2}{8\sigma^2}}e^{\pi \frac{\beta}{b}|\ell|\,M_K} <\infty.
\end{align*}
Thus the series  $V_x(z)$ defined in \eqref{V_x.def.appx} 
converges absolutely and uniformly on compact subsets of $\C$. Since each summand is entire, $V_x$ is entire. Consequently, the function
$$  z\longmapsto \frac{\pi\sigma^2}{b} e^{-\frac{2\pi^2\sigma^2}{b^2}(z-ax)^2} V_x(z)$$
is entire on $\C$ and coincides with
$ t\longmapsto |S_{\check{\varphi}}^{\mathbf A}f(x,t)|^2$ 
for real $t$. Hence $t\mapsto |S_{\check{\varphi}}^{\mathbf A}f(x,t)|^2$ extends to an entire function on $\C$.

\subsection{Proof of Theorem~\ref{local-discrete-error.thm}}\label{proof.local-discrete-error.thm}

Let $p\in[-s,s]$ and $\xi\in[-r,r]$. For $x\in\R$, define
$$W_{x,\xi}(z):=M_f(x,z)e^{-\frac{2\pi i\xi z}{b}},  \ z\in\C,$$
where $M_f(x,z)=|S_{\check{\varphi}}^{\mathbf A}f(x,z)|^2$. 

\smallskip 
We divide the proof into three steps.

\smallskip
\noindent
\textbf{Step 1: Discretization error for the Fourier-type integral.}
Let $x=\frac{\beta}{2}n$ with $n\in[-N,N]\cap\Z$, and write $z=t+iy$. By Lemma~\ref{lem:entire-stlct-magnitude}, we have 
$$M_f(x,z) = \frac{\pi\sigma^2}{b} e^{-\frac{2\pi^2\sigma^2}{b^2}(z-ax)^2} V_x(z), $$
where $V_x(z)=\sum_{\ell\in\Z}\breve r_\ell^x e^{\pi i\frac{\beta\ell}{b}z}$ with coefficients $\breve r_\ell^x$ given in  \eqref{trig.coef.lem.def.eq1} satisfying 
$$ |\breve r_\ell^x| \le \|c\|_\infty^2 \left(1+\frac{\sigma\sqrt{2\pi}}{\beta}\right)e^{-\frac{\beta^2\ell^2}{8\sigma^2}}.$$

Then 
\begin{align}\label{eq:Vx-bound-proof}
|V_x(t+iy)| &\le
\sum_{\ell\in\Z}|\breve r_\ell^x|e^{-\pi\frac{\beta\ell}{b}y} \le \|c\|_\infty^2\big(1+\frac{\sigma\sqrt{2\pi}}{\beta}\big)\sum_{\ell\in\Z}e^{-\frac{\beta^2\ell^2}{8\sigma^2}}
e^{-\pi\frac{\beta\ell}{b}y}\notag\\
	{}& = \|c\|_{\infty}^2 (1 + \frac{\sigma}{\beta} \sqrt{2 \pi}) e^{\frac{2 \pi^2 \sigma^2 y^2}{b^2}} \sum_{\ell \in \Z} e^{- \frac{\beta^2}{8 \sigma^2} (\ell + \frac{4 \pi \sigma^2 y}{b \beta})^2} \notag\\ % \le \|c\|_{\infty}^2 (1 + \frac{\sigma}{\beta} \sqrt{2 \pi}) e^{\frac{2 \pi^2 \sigma^2 y^2}{b^2}} \sum_{\ell \in \Z} e^{- \frac{\beta^2 \ell^2}{8 \sigma^2}} \notag\\
		{}& \le \|c\|_{\infty}^2 (1 + \frac{\sigma}{\beta} \sqrt{2 \pi}) e^{\frac{2 \pi^2 \sigma^2 y^2}{b^2}} \big(1 + \int_{\R} e^{- \frac{\beta^2 s^2}{8 \sigma^2}} ds\big) \notag\\ 
		{}&\le \|c\|_{\infty}^2 \big(1 + \frac{2 \sqrt{2 \pi}\sigma}{\beta} \big)^2 e^{\frac{2 \pi^2 \sigma^2 y^2}{b^2}}.
\end{align}
Therefore,
\begin{eqnarray}\label{eq:Wxy-bound-proof}
|W_{x,\xi}(t+iy)| &=& |M_f(x,t+iy)|\,e^{\frac{2\pi \xi y}{b}}  \notag\\
&\le&
\frac{\pi\sigma^2}{b}\|c\|_\infty^2\left(1+\frac{2\sqrt{2\pi}\sigma}{\beta}\right)^2
e^{-\frac{2\pi^2\sigma^2}{b^2}(t-ax)^2}e^{\frac{4\pi^2\sigma^2 y^2}{b^2}}e^{\frac{2\pi\xi y}{b}}.
\end{eqnarray}
Now consider the strip $U_{\zeta} = \{z \in \C : |\Im(z)| < \zeta\}$, where $\zeta=\frac{b}{2\pi\sigma}$. 
Since \(W_{x,\xi}\) is entire, it is analytic in $U_\zeta$. Moreover, 
\eqref{eq:Wxy-bound-proof} shows that $W_{x,\xi}(z)\to 0$ uniformly as
$|z|\to\infty$ in $U_\zeta$. Also, for $|y|<\zeta$, 
\begin{align}\label{eq:Wxy-integ-bound-proof}
\int_{\R}|W_{x,\xi}(t+iy)|dt &\le \frac{\pi\sigma^2}{b} \|c\|_\infty^2\left(1+\frac{2\sqrt{2\pi}\sigma}{\beta}\right)^2 e^{\frac{4\pi^2\sigma^2 y^2}{b^2}} e^{\frac{2\pi\xi y}{b}}
\int_{\R}e^{-\frac{2\pi^2\sigma^2}{b^2}(t-ax)^2}\,dt \notag\\
&= \sqrt{\frac{\pi}{2}} \sigma \|c\|_\infty^2 \left(1+\frac{2\sqrt{2\pi}\sigma}{\beta}\right)^2
e^{\frac{4\pi^2\sigma^2 y^2}{b^2}}e^{\frac{2\pi\xi y}{b}}.
\end{align}
Since $|y|<\zeta$, we have
$ \frac{4\pi^2\sigma^2 y^2}{b^2}\le 1$ and  $\frac{2\pi\xi y}{b}\le \frac{|\xi|}{\sigma}$.
Hence, 
$$\int_{\R}|W_{x,\xi}(t+iy)|dt \le \sqrt{\frac{\pi}{2}} \sigma \|c\|_\infty^2 \left(1+\frac{2\sqrt{2\pi}\sigma}{\beta}\right)^2 e^{1+\frac{|\xi|}{\sigma}}.$$
Applying Lemma \ref{lem:trapezoidal} gives
\begin{align}\label{eq:disc-error-proof}
\left| I_h^\infty(W_{x,\xi})-\int_{\R}W_{x,\xi}(t) dt \right| \le {}&
\sqrt{2\pi}\sigma \|c\|_\infty^2\left(1+\frac{2\sqrt{2\pi}\sigma}{\beta}\right)^2
\frac{e^{1+\frac{|\xi|}{\sigma}}}{e^{\frac{b}{\sigma h}}-1}.
\end{align}

\smallskip
\noindent
\textbf{Step 2: Cut-off error in the trapezoidal rule.}
We now estimate $ |I_h^\infty(W_{x,\xi})-I_h^H(W_{x,\xi})|$. 
By \eqref{eq:Wxy-bound-proof} with $y=0$, we have 
\begin{align}\label{eq:cutoff-start-proof}
|I_h^\infty(W_{x,\xi})-I_h^H(W_{x,\xi})| &\le
h\frac{\pi\sigma^2}{b} \|c\|_\infty^2 \left(1+\frac{2\sqrt{2\pi}\sigma}{\beta}\right)^2
\sum_{k\in([-H,H]\cap\Z)^c}e^{-\frac{2\pi^2\sigma^2}{b^2}(hk-ax)^2}.
\end{align}
Since $x=\frac{\beta}{2}n$ with $|n|\le N$ and $H=\left\lceil \frac{|a|\beta N}{2h}\right\rceil+q$, 
we have for every $k\in([-H,H]\cap\Z)^c$, 
$$ |hk-ax|= \left|hk-\frac{a\beta}{2}n\right|\ge h(q+|k|-H).$$
Therefore,
\begin{align*}
\sum_{k\in([-H,H]\cap\Z)^c} \hskip-.2in e^{-\frac{2\pi^2\sigma^2}{b^2}(hk-ax)^2}
\le 2 \hskip-.05in\sum_{j=q+1}^\infty e^{-\frac{2\pi^2\sigma^2 h^2}{b^2}j^2}  \le
2 \hskip-.05in\int_q^\infty e^{-\frac{2\pi^2\sigma^2 h^2}{b^2}t^2}\,dt \le
\frac{b}{\sigma h\sqrt{2\pi}} e^{-\frac{2\pi^2\sigma^2 h^2 q^2}{b^2}},
\end{align*}
where  the last inequality follows from the following inequality
    \begin{align*}
        \int_q^{\infty} e^{-p t^{2}}\;dt
	\;\le\; \sqrt{\frac{\pi}{4p}}e^{-p q^{2}}, \ p>0,q\ge0.
    \end{align*}
Substituting this into \eqref{eq:cutoff-start-proof}, we get
\begin{align}\label{eq:cutoff-error-proof}
|I_h^\infty(W_{x,\xi})-I_h^H(W_{x,\xi})|
\le {}& \sqrt{\frac{\pi}{2}} \sigma  \|c\|_\infty^2 \left(1+\frac{2\sqrt{2\pi}\sigma}{\beta}\right)^2
e^{-\frac{2\pi^2\sigma^2 h^2 q^2}{b^2}}.
\end{align}
Combining \eqref{eq:disc-error-proof} and \eqref{eq:cutoff-error-proof}, we obtain
\begin{align}\label{eq:quad-error-proof}
\left| \int_{\R}W_{x,\xi}(t)\,dt-I_h^H(W_{x,\xi}) \right|
\le {}& \sqrt{\frac{\pi}{2}} \sigma \|c\|_\infty^2 \left(1+\frac{2\sqrt{2\pi}\sigma}{\beta}\right)^2
\left(\frac{2e^{1+\frac{|\xi|}{\sigma}}}{e^{\frac{b}{\sigma h}}-1}
+e^{-\frac{2\pi^2\sigma^2 h^2 q^2}{b^2}}\right).
\end{align}

\smallskip
\noindent
\textbf{Step 3: Estimate of \( |f_\xi(p+\xi)-\mathcal G_p(\xi)| \).}
Define 
$$ F_1(n,\xi) := \int_{\R} M_f \left(\frac{\beta n}{2},t\right)e^{-\frac{2\pi i\xi t}{b}}dt$$ 
and 
$$ F_2(n,\xi) :=h\sum_{k=-H}^{H}Y_{n,k}e^{-\frac{2\pi i\xi hk}{b}}.$$ 

By \eqref{eq:continuous-fw} and \eqref{eq:Gp-def}, we have 
\begin{align}
\hskip-.1in|f_\xi(p+\xi)-\mathcal G_p(\xi)|& \le \hskip-.3in \sum_{n \in ([-N , N] \cap \Z)^c} \hskip-.2in | F_1(n , \xi) T_{\frac{\beta}{2} n} \widetilde{\varphi_\xi} (p + \xi)| 
 +  \hskip-.1in\sum_{n = -N}^{N} | (F_1(n , \xi)\hskip-.05in -\hskip-.05in F_2(n , \xi)) T_{\frac{\beta}{2} n} \widetilde{\varphi_\xi} (p + \xi)| \notag\\
         &=: \varepsilon_1 + \varepsilon_2. 
\label{eq:error-split-proof}
\end{align}

We first estimate $\varepsilon_1$. Since 
$ |F_1(n,\xi)| = |\int_{\R} W_{\frac{\beta n}{2}, \xi} (t) dt|$,
%$ |F_1(n,\xi)| \le \int_{\R}M_f\left(\frac{\beta n}{2},t\right)dt$,  
the bound \eqref{eq:Wxy-integ-bound-proof} with $y=0$ gives
$$ |F_1(n,\xi)| \le \sqrt{\frac{\pi}{2}} \sigma \|c\|_\infty^2\left(1+\frac{2\sqrt{2\pi}\sigma}{\beta}\right)^2.$$ 
 By $\widetilde{\varphi_\xi}$ in \eqref{dual.phi_w} and the decay estimate \eqref{eq6}, we have 
\begin{align*}
\varepsilon_1 &\le \sqrt{\frac{\pi}{2}}  \sigma  \|c\|_\infty^2 \left(1+\frac{2\sqrt{2\pi}\sigma}{\beta}\right)^2\hskip-.15in \sum_{n\in([-N,N]\cap\Z)^c}\hskip-.15in\left|T_{\frac{\beta}{2}n}\widetilde{\varphi_\xi}(p+\xi)\right| \\
&\le \sqrt{\pi}\sigma \|c\|_\infty^2 \left(1+\frac{2\sqrt{2\pi}\sigma}{\beta}\right)^2e^{\frac{\xi^2}{4\sigma^2}}K\hskip-.15in\sum_{n\in([-N,N]\cap\Z)^c}\hskip-.15ine^{-\nu\left|p+\frac{\xi}{2}-\frac{\beta}{2}n\right|}.
\end{align*}
Since $ N=\left\lceil \frac{2}{\beta}\left(s+\frac{r}{2}\right)\right\rceil+m,  p\in[-s,s],
|\xi|\le r$, 
we obtain for $n\in([-N,N]\cap\Z)^c$,
\begin{align*}
    \left|p+\frac{\xi}{2}-\frac{\beta}{2}n\right| 
    \ge \left|\frac{\beta}{2}n\right|  - |p| - \left|\frac{\xi}{2}\right| 
    \ge \frac{\beta}{2}(m+|n|-N)
\end{align*}
and therefore
$$\sum_{n\in([-N,N]\cap\Z)^c} e^{-\nu\left|p+\frac{\xi}{2}-\frac{\beta}{2}n\right|}
\le 2\hskip-.05in\sum_{j=m+1}^\infty\hskip-.1in e^{-\frac{\nu\beta}{2}j} \le \frac{4}{\nu\beta}e^{-\frac{\nu\beta}{2}m}.
$$
Hence
\begin{align}\label{eq:eps1-proof}
\varepsilon_1 \le {}& \frac{4\sqrt{\pi}\sigma}{\nu\beta} K \|c\|_\infty^2 e^{\frac{\xi^2}{4\sigma^2}}
\left(1+\frac{2\sqrt{2\pi}\sigma}{\beta}\right)^2 e^{-\frac{\nu\beta}{2}m}.
\end{align}

Since  $Y_{n,k}$ is given by  \eqref{noisy.data.def}, we have 
$$ F_2(n,\xi) = I_h^H(W_{\frac{\beta n}{2},\xi}) + h\sum_{k=-H}^{H}\eta_{n,k}e^{-\frac{2\pi i\xi hk}{b}},  \ n\in[-N,N]\cap\Z. $$
Therefore,
$$ |F_1(n,\xi)-F_2(n,\xi)| \le \left| \int_{\R}W_{\frac{\beta n}{2},\xi}(t) dt - I_h^H(W_{\frac{\beta n}{2},\xi}) \right| + h\|\eta\|_\infty.$$
Substituting this into \eqref{eq:error-split-proof}, we obtain
\begin{align*}
\varepsilon_2 &\le \left( \left| \int_{\R}W_{\frac{\beta n}{2},\xi}(t) dt
- I_h^H(W_{\frac{\beta n}{2},\xi}) \right| + h\|\eta\|_\infty \right)
\sum_{n=-N}^{N} \left|T_{\frac{\beta}{2}n}\widetilde{\varphi_\xi}(p+\xi)\right| \\
&\le \left( \left| \int_{\R}W_{\frac{\beta n}{2},\xi}(t) dt - I_h^H(W_{\frac{\beta n}{2},\xi})
\right| +h\|\eta\|_\infty\right)\sqrt{2}K e^{\frac{\xi^2}{4\sigma^2}}
\sum_{n\in\Z}e^{-\nu\left|p+\frac{\xi}{2}-\frac{\beta}{2}n\right|}.
\end{align*}
Using $ \sum_{n\in\Z}e^{-\nu\left|p+\frac{\xi}{2}-\frac{\beta}{2}n\right|} \le 2\left(1+\frac{2}{\nu\beta}\right)$, 
together with \eqref{eq:quad-error-proof}, we obtain
\begin{align}\label{eq:eps2-proof}
\varepsilon_2 \hskip-.03in \le \hskip-.03in  2\sqrt{2}K e^{\frac{\xi^2}{4\sigma^2}} \hskip-.03in \left(\hskip-.03in 1+\frac{2}{\nu\beta} \hskip-.03in\right) \hskip-.06in
\left[ \hskip-.03in \sqrt{\frac{\pi}{2}} \sigma \|c\|_\infty^2 \hskip-.03in \Bigr( \hskip-.03in 1+\frac{2\sqrt{2\pi}\sigma}{\beta}  \Bigr)^2 
\Bigl(\hskip-.03in \frac{2e^{1+\frac{|\xi|}{\sigma}}}{e^{\frac{b}{\sigma h}}-1} + e^{-\frac{2\pi^2\sigma^2 h^2 q^2}{b^2}} \hskip-.03in \Bigr) \hskip-.06in + h\|\eta\|_\infty \hskip-.03in\right]\hskip-.03in.
\end{align}

Finally, combining \eqref{eq:error-split-proof}, \eqref{eq:eps1-proof}, and \eqref{eq:eps2-proof}, and using \(|\xi|\le r\), we arrive at
\begin{align*}
|f_\xi(p+\xi)-\mathcal G_p(\xi)|& \le Ke^{\frac{\xi^2}{4\sigma^2}} \left( \frac{4\sqrt{\pi}\sigma}{\nu\beta}\|c\|_\infty^2\Big(1+\frac{2\sqrt{2\pi}\sigma}{\beta}\Big)^2e^{-\frac{\nu\beta}{2}m} 
+2\sqrt{2}h\Big(1+\frac{2}{\nu\beta}\Big)\|\eta\|_\infty \right.
\nonumber\\
&\left.+2\sqrt{\pi}\sigma\|c\|_\infty^2
\Big(1+\frac{2\sqrt{2\pi}\sigma}{\beta}\Big)^2 \big(1+\frac{2}{\nu\beta}\big)
\Big(\frac{2e^{1+\frac{|\xi|}{\sigma}}}{e^{\frac{b}{\sigma h}}-1}
+e^{-\frac{2\pi^2\sigma^2 h^2 q^2}{b^2}}\Big)\right),
\end{align*}
which completes the proof.

\end{appendices}

\end{document}